\documentclass{mcom-l}

\usepackage{stackrel}
\usepackage{amsmath}
\usepackage{amsfonts}
\usepackage{amssymb}
\usepackage{color}
\usepackage[a4paper]{geometry}
\usepackage{graphicx}%
\providecommand{\U}[1]{\protect\rule{.1in}{.1in}}
\numberwithin{equation}{section}

\newcommand{\jint}{j = 1, \ldots, N-1}
\newcommand{\jsub}{j = 1, \ldots, N}

\newcommand{\igg}{\color{black}}

\newcommand{\imag}{\color{black}}
\newcommand{\imagg}{\color{black}} 

\newcommand{\Cstab}{C_{\mathrm{stab}}}

\newcommand{\Ctrace}{C_{\mathrm{trace}}}

\newcommand{\Var}{\mathrm{Var}}

\newcommand{\cH}{\mathcal{H}}

\newcommand{\amax}{a_{\max}}

\newcommand{\amin}{a_{\min}}

\newcommand{\betamax}{\beta_{\max}}

\newcommand{\cmin}{c_{\min}}

\newcommand{\cmax}{c_{\max}}

\newtheorem{theorem}{Theorem}[section]

\newtheorem{assumption}[theorem]{Assumption}
\newtheorem{discussion}[theorem]{Discussion}

\newtheorem{corollary}[theorem]{Corollary}

\newtheorem{definition}[theorem]{Definition}

\newtheorem{lemma}[theorem]{Lemma}
\newtheorem{notation}[theorem]{Notation}

\newtheorem{proposition}[theorem]{Proposition}
\newtheorem{remark}[theorem]{Remark}

\begin{document}

\title[Stability and  error analysis for the Helmholtz equation with
variable coefficients]{Stability and finite element error analysis for the Helmholtz equation with
variable coefficients}
\author{I.G. Graham}
\address{Department of Mathematical Sciences, University of Bath,
Bath BA2 7AY, United Kingdom}
\email{i.g.graham@bath.ac.uk}
\author{S.A. Sauter}
\address{Institut f\"{u}r Mathematik, Universit\"{a}t Z\"{u}rich,
Winterthurerstrasse 190, CH-8057 Z\"{u}rich, Switzerland} 
\email{stas@math.uzh.ch}
\maketitle

\begin{abstract}
{ We discuss the stability theory and numerical analysis of the  Helmholtz equation  with
variable and possibly non-smooth or oscillatory coefficients.
Using the unique continuation principle and the Fredholm alternative,  we
first give an existence-uniqueness result for this problem, which holds under
rather general conditions on the coefficients and on the domain. 
Under additional assumptions, we derive estimates for the stability constant
(i.e., the norm of the solution operator) in terms of the data (i.e. PDE
coefficients and frequency), and we apply these estimates to obtain a new  finite
element error analysis for the Helmholtz equation which is valid at high
frequency and with variable wave speed. The central role played by the
stability constant in this theory leads us to investigate its behaviour with
respect to coefficient  variation in detail. We give, via a 1D analysis,
an {\em a priori}  bound with    stability constant  growing  exponentially in the variance of the coefficients (wave
speed and/or  diffusion coefficient).  Then, by means  a family of analytic examples (supplemented by
numerical experiments),  we  show that this estimate is sharp. }


\end{abstract}





\section{Introduction}


In this paper we consider the Helmholtz equation
\begin{equation}
-\operatorname*{div}\left(  a\operatorname*{grad}u\right)  -\left(
\frac{\omega}{c}\right)  ^{2}u=f \label{eq:Helm}%
\end{equation}
on a bounded connected Lipschitz domain $\Omega\subset\mathbb{R}^{d}$, \ $d=1,2,3$, with
angular frequency $\omega\geq\omega_{0}>0$, given data $f\in L^{2}(\Omega)$,
and { real-valued} scalar coefficients $a,c$ which are allowed to vary
spatially,
but which will be assumed to be
bounded above and below by strictly positive numbers. The problem
(\ref{eq:Helm}) is supplemented with mixed boundary conditions on
$\Gamma:=\partial\Omega$ of the form
\begin{equation}
a\frac{\partial u}{\partial{n}}-\mathrm{i}\omega\beta u=g\quad\text{on }%
\quad\Gamma_{{N}},\quad u=0\quad\text{on}\quad\Gamma
_{\operatorname*{D}}\ , \label{eq:ImpBC}%
\end{equation}
for given $g\in H^{-1/2}(\Gamma_{{N}})$ and{ real-valued}
$\beta\in L^{\infty}(\Gamma_{{N}})$.
Here $\Gamma_{{N}},\ \Gamma_{{D}}$ are relatively
open pairwise disjoint subsets of $\Gamma$, with $\Gamma=\overline
{\Gamma_{{N}}}\cup\overline{\Gamma_{{D}}}$ and
$\partial/\partial{n}$ denotes the outward normal derivative.  {\igg In \eqref{eq:ImpBC}, the set  $\mathrm{supp} (\beta) \subseteq \overline{\Gamma_N} $ is required to have positive $d-1$ dimensional measure. }

For the strong formulation (\ref{eq:Helm}), {\igg \eqref{eq:ImpBC}},  a standard, additional requirement
is that $a$ is Lipschitz continuous, so that its gradient is defined almost
everywhere. We shall consider rougher coefficients $a,c$, via
the weak form: seek $u\in\mathcal{H}$ to satisfy
\begin{equation}
B_{a,c}(u,v)\ :=\ \int_{\Omega}\left(  a\nabla u \cdot \nabla\overline{v}-\left(
\frac{\omega}{c}\right)  ^{2}u\overline{v}\right)  \ -\ \mathrm{i}\omega
\int_{\Gamma}\beta u\overline{v}\ =\ \int_{\Omega}f\overline{v}+\int_{\Gamma
}g\overline{v}\ , \label{eq:weak}%
\end{equation}
for all $v\in\mathcal{H}$, where $\mathcal{H}$ denotes the functions in
$H^{1}(\Omega)$ with vanishing trace on $\Gamma_{D}$.
Problem \eqref{eq:weak} arises as a frequency domain scalar approximation of
certain elastic or electromagnetic propagation and scattering model problems
(see \S \ref{subsec:phys}).
In this context, the behaviour of $u$ as frequency $\omega$ grows is of both
physical and numerical importance and occupies considerable contemporary interest.

When $a$ and $c$ are constant and (for simplicity) $\beta=1$, it is well-known
that the \emph{a priori} energy bound
\begin{equation}
\left(  \int_{\Omega}a|\nabla u|^{2}\ +\ \left(  \frac{\omega}{c}\right)
^{2}\left\vert u\right\vert ^{2}\right)  ^{1/2}\ \leq\ C_{\mathrm{stab}%
}\left(  \Vert f\Vert_{L^{2}(\Omega)}^{2}+\Vert g\Vert_{L^{2}(\Gamma)}%
^{2}\right)  ^{1/2}\ \label{eq:apriori}%
\end{equation}
holds for some $C_{\mathrm{stab}}>0$ independent of $f,g$ and $\omega$,
provided $\Omega$ is Lipschitz and star-shaped with respect to a ball -- see
\cite[Prop 8.1.4]{MelenkDiss} (for $d=2$), \cite[Theorem 1]{cummings-feng06}
($d=3$) and e.g. \cite{GaGrSp:15} for a recent survey.

{ A standard approach to showing well-posedness of the problem
\eqref{eq:weak} is to first establish the \emph{a priori} bound
\eqref{eq:apriori} -- which provides uniqueness -- and then to infer existence
via the Fredholm alternative. This approach restricts the range of problems
which can be treated to those for which an \emph{a priori} bound is available.
A more general well-posedness theory can be obtained via the unique
continuation principle (UCP) - and used in \cite[Chap. 4.3]{Leis86} (also
\cite[Rem. 8.1.1]{MelenkDiss}). Our first contribution is to present
well-posedness results for \eqref{eq:Helm}, \eqref{eq:ImpBC} which do not
require the a priori bound \eqref{eq:apriori}. This is done by appealing to
the literature on the UCP and applying it to the special case of the Helmholtz
equation.}

{ Nevertheless, estimates like \eqref{eq:apriori}, when they
are available, play a crucial role in the numerical analysis of
\eqref{eq:weak}.
For the conforming {\imag fixed-order}  Galerkin finite element approximation of
\eqref{eq:weak} with constant coefficients, {\imag a mesh diameter condition  of the form
{ $h\omega^{2}C_{\mathrm{stab}}$} sufficiently small} 
appears as a 
condition for quasi-optimality \cite{mm_stas_helm2}.
{\igg Our second contribution is
to extend this  theory to a class  of variable
coefficients, highlighting the role played by the variable-coefficient
analogue of \eqref{eq:apriori}. In order to do this we extend
\eqref{eq:apriori} to the case of variable wavespeed (requiring that the
wavespeed should be essentially non-oscillatory), and then give an analysis of
quasi-optimality of conforming finite element methods for \eqref{eq:weak} for this class of coefficients. }

{\igg The requirement  that   $h  \omega^{2}$  should be sufficiently small   is a symptom of  the so-called pollution effect (cf.
{\cite{BaSa:00}), which says that a bounded number of grid points per wavelength
  is, in general, \textit{not} sufficient
for a low order Galerkin finite element method to be  quasi-optimal.
There are many publications on the pollution
effect and on how to cure it, either using a  refined error analysis and/or more
sophisticated discretizations (e.g. \cite{MelenkSauterMathComp, mm_stas_helm2, MPS13,Wua,Wub,Wuc}). All these
approaches rely on refined \textquotedblleft split\textquotedblright%
\ regularity estimates for the adjoint Helmholtz problem and we are not aware
of such results for the class of coefficients which we consider in this paper.
It is thus  an open question whether,  for
the class of coefficients considered here,  the property (at least) of discrete stability can be proved under a more 
relaxed condition on $h$. In  \cite{Wub} it was recently shown that discrete stability (and a corresponding error estimate) holds under the weaker condition   $h^{2}\omega^{3}\lesssim O\left(  1\right)  $, but only for 
 coefficients
which are a very small perturbation of a  constant function.   }}

In the final part of the paper we investigate the outcome when the
``non-oscillatory'' assumption on the coefficients is violated. Here our
contribution is a detailed 1D analysis which shows that the estimate
\eqref{eq:apriori} continues to hold, but the stability constant
$C_{\mathrm{stab}}$ has an estimate which {\imagg can blow up exponentially in the
variance of $a$ or $c$ (or both) for certain parameter configurations}. Investigating this further, we present a
class of examples in which this exponential blow-up is realised.

{This paper is organized as follows. In
\S \ref{SecAbsExUnique}, we present the existence-uniqueness of (\ref{eq:Helm}%
)
allowing low regularity of the coefficients $a$ and $c$ via the UCP.
The main result in \S \ref{SecAbsExUnique} is Theorem \ref{thm:wp}, which
requires
that $a,c$ are uniformly positive and
$a,c^{-2}\in L^{\infty}\left(  \Omega\right)  $ (for $d\leq2$). For $d\geq3$,
the slightly more restrictive assumption that $a$ has a $C^{0,1}$ extension
from $\Omega$ to a certain extended domain appears. Moreover the latter
condition on $a$ is sharp in 3D, as the \textquotedblleft
counter-examples\textquotedblright\ in Remark \ref{Counterex} imply. }

Turning to the \emph{a priori } bound \eqref{eq:apriori}, the standard
approach is
to use the ``Rellich'' test function $v=\mathbf{x} \cdot \nabla u$ (or linear
combinations of $v$ and $u$ - commonly known as the ``Morawetz multiplier'' )
in the weak form \eqref{eq:weak} to obtain stability estimates - see e.g.
\cite{Sp:13} and a number of authors have recently applied this technique in the
variable coefficient case (see discussion of literature below).
In Section \ref{SecDleq2} we illustrate the outcome applying this analysis in
the case of variable $c$ (but restricting to $a=1$ for illustration)
and we explain how this approach yields the bound \eqref{eq:apriori}, but only
under the quite restrictive condition

{
\begin{equation}
\frac{\mathbf{x}\cdot \nabla c}{c}\leq1-\theta\ , \label{eq:restrictive}%
\end{equation}
for some $\theta>0$. Thus, although $c$ is allowed to decrease arbitrarily
quickly along the radial directions, there is a severe restriction on any
increase in $c$ and hence oscillatory behaviour for $c$ is essentially ruled
out.
}

Section {\ref{SecGalDisc} is devoted to the
convergence and stability analysis of a conforming Galerkin discretization. As
it is common for elliptic problems with compact perturbations, the
discretization has to satisfy a \textquotedblleft minimal resolution
condition\textquotedblright\ for the \textit{adjoint approximability constant}
(see (\ref{defhetapproxconst})) and then quasi-optimal error estimates
follow. In Theorem \ref{TheoStabDisc} we prove that the quasi-optimality
constant for the energy error is independent of the wave number and give the
explicit dependence on the coefficients }$a$ and $c$. The adjoint
approximability is estimated in Theorem {\ref{thm:hetconv}; the estimate is
explicit in the wave number }$\omega$, {reciprocal density }$a$, the wave
speed $c$, the mesh size $h$ of the finite element mesh, and the stability
constant $C_{\mathrm{stab}}$. This shows the importance of
coefficient-explicit estimates of $C_{\mathrm{stab}}$ in order to obtain a
coefficient-explicit minimal resolution condition and error estimate. In the
setting of \S \ref{SecDleq2} such estimates for $C_{\mathrm{stab}}$ are available.

{ In \S \ref{Sec1DCase}, we consider one-dimensional problems
with wave speed $c$ and diffusion coefficient $a$ which may both be spatially
oscillatory and suffer jumps. We study how this affects $C_{\mathrm{stab}}$.
We prove the estimate \eqref{eq:apriori}  with  $C_{\mathrm{stab}}$ growing    exponentially in the
\textit{variation} of the coefficients $a,c$. We conclude \S \ref{Sec1DCase} with
some analytic and numerical calculations on a particular example where the
exponential growth of $C_{\mathrm{stab}}$ is realised, demonstrating the
sharpness of the theoretical estimates. In \cite{Torres17}, \cite{Torres18}
a refined analysis is
presented for \textit{perfectly oscillating} coefficients $c$. Here, although
the variation increases with the oscillation, it is shown that in this case
the stability constant is bounded independently of the oscillation. }



\subsection{Applications}

\label{subsec:phys}

In seismic inversion \textquotedblleft the forward model\textquotedblright\ is
the elastic wave equation and the (generally spatially dependent) elastic
parameters - density and Lam\'{e} parameters - have to be recovered in the
inversion process (so called \textquotedblleft full waveform
inversion\textquotedblright). The full forward model is thus a 3D system of
evolution equations, which can also be converted to a Helmholtz-like system in
the frequency domain by Fourier transform. The scalar equation \eqref{eq:Helm}
can be obtained as an approximation of this system where it is known as
\textquotedblleft the acoustic approximation\textquotedblright. In this case
$a=1/\rho$ and $c^{2}=\rho c_{P}^{2}$, where $\rho$ is the density and $c_{P}$
is the speed of longitudinal waves (also called $P$-waves) in the medium.
(See, e.g. \cite{Ta:84}.)

In the theory of photonic crystals, the spectra of Maxwell operators on
infinite periodic media are analysed by the Floquet transform.
The Maxwell system can be decomposed into TM (transverse magnetic) and TE
(transverse electric) components and the computation of each case reduces to a
scalar problem of form \eqref{eq:Helm}, with the TM mode having $a$ constant
and $c$ variable and the TE mode having $c$ constant and $a$ variable. A key
reference is Kuchment \cite{Ku:03}.

\subsection{Literature on this topic}

\label{subsec:literature}


The numerical analysis of heterogeneous Helmholtz problems can be traced back
at least to Aziz, Kellogg and Stephens \cite{azizhelmholtz}. Here problem
\eqref{eq:Helm} in 1D with $a=1$ and $c$ positive (and at least $C^{1}$),
together with mixed Dirichlet-Impedance boundary conditions, was considered. A
decomposition of the solution $u$ in terms of explicitly oscillating functions
$\exp(\pm\mathrm{i}\omega K(x))$ (with $K$ denoting an antiderivative of
$1/c$) and with smooth non-oscillatory multiplicative factors was obtained. In
this analysis, a test function (of \textquotedblleft
Morawetz-type\textquotedblright\ (cf. \cite{MorwaetzLudwig})) of the form
$v=au^{\prime}+bu$, with $a$ and $b$ chosen as solutions of a certain ODE was
used.

In \cite{MaIhBa:96}, the \textquotedblleft Rellich-type\textquotedblright%
\ test function $v\left(  x\right)  =xu^{\prime}\left(  x\right)  $
(\cite{Rellich_trick}) was used to prove stability bounds for a fluid-solid
interaction problem, modeled in 1D with piecewise constant material properties
having discontinuities at two points. This problem has some resemblance to
\eqref{eq:Helm} with jumping coefficients $a$ and $c$, but the interface
conditions at the jump points are different. Nevertheless \cite{MaIhBa:96}
provided a frequency-independent stability result for this problem.

In the very interesting recent PhD thesis \cite{FreletDiss} (see also \cite{BaChGo:17}) the technique of
\cite{MaIhBa:96} was adapted to \eqref{eq:Helm} in 1D with piecewise constant
coefficients, allowing an arbitrary number of jumps of arbitrary magnitude. A
frequency-independent stability estimate is proven with a constant explicit in
the number and magnitude of the jumps. This can grow exponentially with
respect to the number of jumps. This is a special case of our results obtained in \S 
\ref{subsec:stab}.  

{An early paper on the use of multiplier techniques to prove stability for the
Helmholtz equation with variable refractive index is \cite{PerthameVega},
where their conditions (1.7), (1.8) are similar to ours. This technique can
directly be generalized to variable diffusion (see, e.g. \cite{Brownetal})} in
general dimension, essentially adding a condition on $a$ to
\eqref{eq:restrictive} and allowing both $c$ and $a$ to vary but not to oscillate.

Related results are in the recent preprint \cite{OhVe:16}, in which a certain
Helmholtz transmission problem is considered, corresponding to scattering from
a homogenised multiscale material in 2D. Estimates of the form
\eqref{eq:apriori} are obtained, but with a stability constant which grows
{\imagg cubically in $\omega$.}

An alternative point of view is presented in \cite{FeLiLo:14}. This paper is
presented in the context of acoustic scattering in random media, but the
results can be restated so that they apply to deterministic media modelled by
\eqref{eq:Helm} with $a=1$ and $c$ variable. They show that if
$1/c=1+\varepsilon\eta$ where $\Vert\eta\Vert_{L^{\infty}\left(
\Omega\right)  }\leq1$ then wavenumber independent stability can be proved
provided a frequency-dependent $\varepsilon$ is chosen with $\varepsilon
=\mathcal{O}(\omega^{-1})$. Obviously this condition is very restrictive on
the range of allowable wave speeds but the result is interesting in that it
allows very rough perturbations of a constant speed. Such rough perturbations
are forbidden by \eqref{eq:restrictive}. This result is exploited in an
uncertainty quantification context in \cite{FeLiLo:14}. 
A condition which is very similar to \eqref{eq:restrictive} already
appeared in \cite{PerthameVega}; generalisations thereof in the case of
physically relevant interior and exterior Helmholtz scattering problems in
heterogeneous and stochastic media are discussed in detail in \cite{GrPeSp:17}.


There is considerable current interest in linear algebra problems arising from
the discretization of heterogeneous wave problems (e.g.
\cite{Ganesh_Helm_horn}, \cite{GrSpVa:16}). The stability theory of the
underlying PDE turned out to be key to rigorously understanding the
performance of iterative methods in the homogeneous case (e.g.
\cite{GaGrSp:15}, \cite{GrSpVa:17}), and so analogous results for the
heterogeneous case will again be important in the construction and analysis of
efficient solvers.


In obstacle scattering in homogeneous media, the stability of the solution
operator is often associated with the scattering boundary being
\textquotedblleft non-trapping\textquotedblright, and the introduction of
trapping boundaries (e.g. boundaries with cavities) causes the norm of the
solution operator to blow up (e.g. \cite{Graham_ActaNumerica_2012},
\cite{Betcke_cond_numb}). In some sense the imposition of the condition
\eqref{eq:restrictive} can be thought of as a condition which removes the
possibility of trapped waves. This correspondence makes more sense in exterior
scattering problems and is explored in more detail in \cite{GrPeSp:17}. 
A
condition closely related to \eqref{eq:restrictive} arises in
\cite{Ganesh_SignDef}. 
The role of heterogeneity in the far field is discussed
in \cite{PerthameHelm}. Heterogeneity is also discussed in the microlocal
analysis literature, where necessarily coefficients are assumed to be smooth.
A very general result which shows that the stability constant can grow at
worst exponentially in $\omega$ for $C^{\infty}\left(
\Omega\right)  $ coefficients is given in \cite{Bu:16}. For an analogous
result for the transmission problem see \cite{Be:03}. 

Finally we note that wave propagation and scattering in heterogeneous and
random media is of considerable interest in the inverse problems and imaging
community. There, the Green's function for the heterogeneous Helmholtz problem
- while not known analytically - is an important object of study.
For certain coefficient configurations it is possible to analyze its
qualitative behavior, derive expansions, as well as explicit dependencies on
certain parameters - see, e.g., \cite{Ammari_variable},
\cite{Garnier_papnicolau_Helm_green}, although the issues pursued there are
somewhat different to the topic of this paper.

\subsection{Some notation\label{SecSetting}}
For $s\geq0$, $1\leq p\leq\infty$, $W^{s,p}\left(  \Omega\right)  $ will
denote the classical Sobolev spaces of complex-valued functions with norm
$\left\Vert \cdot\right\Vert _{W^{s,p}\left(  \Omega\right)  }$. Sobolev
spaces on the boundary $\Gamma$ of $\Omega$ are defined in the usual manner
and are denoted by $W^{s,p}\left(  \Gamma\right)  $ and $H^{s}\left(
\Gamma\right)  $. As usual we write $L^{p}\left(  \Omega\right)  $ instead of
$W^{0,p}\left(  \Omega\right)  $ and $H^{s}\left(  \Omega\right)  $ for
$W^{s,2}\left(  \Omega\right)  $. The scalar product and norm in $L^{2}\left(
\Omega\right)  $ and $L^{2}\left(  \Gamma\right)  $ are denoted respectively
by
\[%
\begin{array}
[c]{llll}%
\left(  u,v\right)  :=\int_{\Omega}u\bar{v} & \text{and} & \left\Vert
u\right\Vert :=\left(  u,u\right)  ^{1/2} & \text{in }L^{2}\left(
\Omega\right)  ,\\
\left(  u,v\right)  _{\Gamma}:=\int_{\Gamma}u\bar{v} & \text{and} & \left\Vert
u\right\Vert _{\Gamma}:=\left(  u,u\right)  _{\Gamma}^{1/2} & \text{in }%
L^{2}\left(  \Gamma\right)  .
\end{array}
\]
For a (topological) vector space $V$, we denote its topological dual space
(i.e. all bounded linear functionals on $V$) by $V^{\prime}$ and we denote its
anti-dual (all bounded anti-linear functionals on $V$) by $V^{\times}$.

For parameters $\alpha_{0}<\alpha_{1}$ and any subspace $V\left(
\Omega\right)  \subseteq L^{\infty}\left(  \Omega,\mathbb{R}\right)  $, we
define
\[
V\left(  \Omega,\left[  \alpha_{0},\alpha_{1}\right]  \right)  :=\left\{  w\in
V\left(  \Omega\right)  :\alpha_{0}\leq\underset{x\in\Omega
}{\operatorname*{ess}\inf}\, w\left(  x\right)  \leq\underset{x\in
\Omega}{\operatorname*{ess}\sup}\, w\left(  x\right)  \leq\alpha_{1}\right\}
,
\]
and in doing so we implicitly assume $-\infty< \alpha_{0} < \alpha_{1} <
\infty$.

In what follows we shall
study the Helmholtz problem \eqref{eq:weak} and its adjoint. In abstract form
these read:
\begin{equation}
\text{Seek}\quad u\in\mathcal{H}:\quad B_{a,c}\left(  u,v\right)  =F\left(
v\right)  \qquad\text{for all}\quad v\in\mathcal{H}\text{.} \label{HelmAbs}%
\end{equation}
\begin{equation}
\text{Seek}\quad z\in\mathcal{H}:\quad B_{a,c}\left(  v,z\right)  =G\left(
v\right)  \qquad\text{for all}\quad v\in\mathcal{H}\text{.} \label{DualAbs}%
\end{equation}
where $F\in\mathcal{H}^{\times}$ and $G\in\mathcal{H}^{\prime}$ are given.
Throughout we shall make the basic assumptions that $a,c$ and $\beta$ are
real-valued and:

\[
a\in L^{\infty}\left(  \Omega,\left[  a_{\min},a_{\max}\right]  \right)
,\quad c\in L^{\infty}\left(  \Omega,\left[  c_{\min},c_{\max}\right]
\right)  ,\quad\beta\in L^{\infty}\left(  \Gamma,\left[  0,\beta_{\max
}\right]  \right)  ,
\]
for some positive $\amin, \cmin, \beta_{\max}$. 
Additional assumptions will be introduced where needed. We also assume that
the frequency $\omega$ is bounded away from $0$,
\[
\omega\geq\omega_{0}>0.
\]
{For $0<\lambda<1$, we denote by $C^{0,\lambda}\left(  \overline{\Omega
}\right)  $ the space of H\"{o}lder continuous functions with H\"{o}lder
exponent $\lambda$.}




\section{Well-Posedness via Unique Continuation\label{SecAbsExUnique}}

In this section we generalize the original ideas in \cite{MelenkDiss} and
apply the Unique Continuation Principle (UCP) to obtain uniqueness for
problems \eqref{HelmAbs}, \eqref{DualAbs}, under rather general conditions on
the coefficients $a$, $c$ and $\beta$. This, combined with the Fredholm
Alternative, proves well-posedness for these problems.
In our first result - Theorem \ref{thm:UCP} - we collect what is known about
the UCP from various references in a form which should be useful to numerical
analysts working on Helmholtz problems. Then, we prove the well-posedness in
Theorem \ref{thm:wp}. To our knowledge this procedure provides Helmholtz
well-posedness in the most general framework possible.  We
emphasize that the arguments and reasoning for the proof of well-posedness in
this section are original but were made available to the authorship of
\cite{GrPeSp:17} and anticipated therein.

Another approach is to first obtain \emph{a priori} bounds which in turn imply
uniqueness. However, as we explain below, up to current knowledge, a priori
bounds require quite strong conditions on the coefficients. 

\begin{theorem}
\label{thm:UCP} Suppose $a\in L^{\infty}(\Omega,[a_{\min},a_{\max}])$ is
real-valued with $0<a_{\min}$, and suppose $\kappa\in L^{p}(\Omega)$, for some
$p>1$. In addition, when $d\geq3$ we require
\begin{equation}
a\in C^{0,1}(\overline{\Omega})\quad\text{and}\quad\kappa\in L^{d/2}%
(\Omega)\ . \label{251}%
\end{equation}
Let $u\in\mathcal{H}$ satisfy:
\begin{equation}
\int_{\Omega}\left\{  a\nabla u\cdot\nabla\overline{v}+\kappa u\overline
{v}\right\}  \ =\ 0\quad\text{for all}\quad v\in\mathcal{H}\ . \label{252}%
\end{equation}
Then, if $u$ vanishes on a ball $B$ of positive radius, with $\overline
{B}\subset\Omega$, it follows that $u$ vanishes identically on $\Omega$.
\end{theorem}

%

\proof
{\imag We start with the case $d \geq 3$.}  First note that  \eqref{252} implies
\[
-\nabla\cdot a\nabla u+\kappa u=0,\quad\text{almost everywhere on }\Omega.
\]
In the special case when $a=1$, we have $\Delta u=\kappa u$, and the result
follows from \cite[Theorem 6.3]{KenigUCP}. In the general case when $a\in
C^{0,1}(\overline{\Omega})$, Rademacher's theorem implies $\nabla
a\in(L^{\infty}(\Omega))^{d}$ and so we can write {\imag (cf.  \cite[Theorem 8.8] {GiTr:83})}
\[
-\Delta u=a^{-1}(\nabla a\cdot\nabla u-\kappa u),\quad\text{almost everywhere
on }\quad\Omega.
\]
Since $u\in\mathcal{H}\subset H^{1}(\Omega)$, the Sobolev embedding theorem
gives $u\in L^{\frac{2d}{d-2}}(\Omega)$ and using $\kappa\in L^{d/2}(\Omega)$
and H\"{o}lder's inequality, we obtain {\imag $\kappa u\in L^{p}(\Omega)$}, 
where
$p={{2d}/{(d+2)}=2-4/{(d+2)}}<2$. Moreover
\[
|\nabla a\cdot\nabla u|\leq|\nabla a|\,|\nabla u|.
\]
Then, local elliptic regularity estimates for the Laplace operator (for
example \cite[Theorem 9.11]{GiTr:83}) imply that $u\in W_{\mathrm{loc}}%
^{2,p}(\Omega)$, The result then follows from \cite[Theorem 1]{Wolff_ucp},
with $A=a_{\min}^{-1}|\kappa|$ and $B=a_{\min}^{-1}|\nabla a|$.

{The cases }$d=1,2$ are somewhat easier. When $d=2$ the proof can be found in
\cite[Theorem 1.1]{Alessandrini_ucp}. For $d=1$ the proof is very similar. If
$\kappa=0$, it is elementary: Let $\sigma\subset\Omega$ be a connected subset
on which $u=0$. Let $\xi,\zeta\in\sigma$ with $\xi\not =\zeta$. Then, the
solution of $-\left(  au^{\prime}\right)  ^{\prime}=0$ can be written in the
form $u\left(  x\right)  =C\int_{\xi}^{x}\left(  \frac{1}{a\left(  s\right)
}ds\right)  $ for all $x\in\Omega$. Since $u\left(  \zeta\right)  =0$, we have
$C=0$ and thus $u=0$. If $\kappa\in L^{p}\left(  \Omega\right)  $ for $p>1$
one applies the theory as in \cite[Sec. 2]{Alessandrini_ucp} and observes that
the Sobolev embedding theorems used in the proofs apply also for $d=1$. This
allows us to transform the equation $-\left(  au^{\prime}\right)  ^{\prime
}+\kappa u=0$ in $\Omega$ as explained in \cite[Sec. 3]{Alessandrini_ucp} to a
local equation in divergence form $-\left(  \hat{a}v^{\prime}+\hat{\kappa
}v\right)  ^{\prime}=0$, where $\hat{a}\in L^{\infty}\left(  \Omega,\left[
\hat{a}_{0},\hat{a}_{1}\right]  \right)  $ for some $0<\hat{a}_{0}\leq\hat
{a}_{1}<\infty$ and $\hat{\kappa}\in L^{t}$ for some $1<t<p$. The variation of
constant formula leads to%
\[
v\left(  x\right)  =C\left(  \int_{\xi}^{x}\frac{1}{\hat{a}\left(  t\right)
}\exp\left(  \int_{\xi}^{t}\frac{\hat{\kappa}\left(  s\right)  }{\hat
{a}\left(  s\right)  }ds\right)  dt\right)  \exp\left(  \int_{\xi}^{x}%
-\frac{\hat{\kappa}\left(  s\right)  }{\hat{a}\left(  s\right)  }ds\right)
\]
and one uses the same argument as in the case $\kappa=0$ to prove $v=0$.
\endproof

\begin{remark}
\label{Counterex}For $d\geq3$, the condition $a\in C^{0,1}\left(
\overline{\Omega},\mathbb{R}\right)  $ in Theorem \ref{thm:UCP}
is sharp in the following sense.
There exists a bounded domain $\Omega\subset\mathbb{R}^{d}$, a coefficient
$a\in%
{\displaystyle\bigcap\limits_{\lambda<1}}
C^{0,\lambda}\left(  \overline{\Omega}\right)  $, a parameter $\omega>0$, and
a function $u\in C^{\infty}\left(  \Omega\right)  $ with $\overline
{\mathrm{supp}\,u}\subsetneqq\Omega$ with
\[
-\operatorname*{div}\left(  a\nabla u\right)  +\omega^{2}u=0\qquad\text{in
}\Omega.
\]
A constructive proof is given in \cite{Filonov}. In \cite[Rem. 6.5]{KenigUCP}
an example is constructed which shows that also the assumption on $\kappa$ as
in (\ref{251}) is sharp.
\end{remark}

Before we prove our main result - Theorem \ref{thm:wp} - we need to state an
assumption which will be required.


\begin{assumption}
\label{ass:UCP}\quad

\begin{enumerate}
\item Either $\beta\in L^{\infty}(\Gamma_{N},[0,\beta_{\mathrm{max}}])$ or
$\beta\in L^{\infty}(\Gamma_{N},[-\beta_{\mathrm{max}},0])$, with
$\beta_{\mathrm{max}}>0$ and the set
{\igg $\gamma:= \mathrm{supp}(\beta)\subset\overline{\Gamma_{N}}$}
has positive $d-1$ dimensional measure.

\item There exists a bounded connected domain $\Omega^{\ast}\supseteq\Omega$
with the property that
\[
\Gamma\backslash\gamma\subset\partial\Omega^{\ast}\quad\text{and}\quad
\Omega^{\ast}\backslash\Omega\quad\text{has positive $d$ dimensional
measure}.
\]

\item The coefficients $a$ and $c$ have extensions $a^{\ast},c^{\ast}$ to
$\Omega^{\ast}$ with $a^{\ast}\in L^{\infty}(\Omega^{\ast},[a_{\min}^{\ast
},a_{\max}^{\ast}])$ and $c^{\ast}\in L^{\infty}(\Omega^{\ast},[c_{\min}%
^{\ast},c_{\max}^{\ast}])$ with $a_{\min}^{\ast}>0$ and $c_{\min}^{\ast}>0$.

\item When $d\geq3$ we require in addition $a^{\ast}\in C^{0,1}(\overline
{\Omega^{\ast}})$.

\begin{figure}[ptb]
\begin{center}
\includegraphics[
height=2.0531in,
width=2.3705in
]{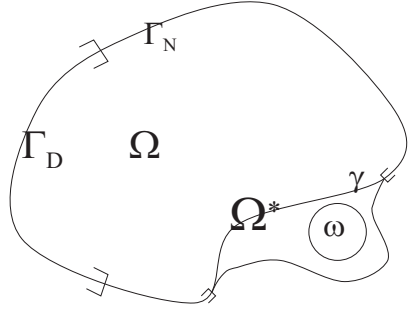}
\end{center}
\caption{Illustration of the geometric construction in Assumption
\ref{ass:UCP} for the unique continuation principle (UCP).}%
\label{Figucp}%
\end{figure}
\end{enumerate}
\end{assumption}

{\igg  Parts (1) and (2) of this assumption are illustrated in Figure \ref{Figucp}.
  Note that parts  (2) and (3)  are not restrictive
for bounded Lipschitz domains but merely stated to introduce the constants
$a_{\min}^{\ast}$ and $c_{\min}^{\ast}$ (cf. Cor. \ref{Cora=1} and its proof).
The assumption is key for proving uniqueness of the heterogeneous Helmholtz
equation via the unique continuation principle.}

{\igg \begin{theorem}
  \label{thm:wp} Suppose Assumption \ref{ass:UCP} is satisfied.
  Then the following hold.
  \begin{itemize}
  \item[(i)]
    $B_{a,c}$ is a bounded sesquilinear form, with    
 \begin{equation}
 |B_{a,c}(u,v)|\ \leq\ C_{a,c}\Vert u\Vert_{\mathcal{H},a,c}\Vert
 v\Vert_{\mathcal{H},a,c}\ ,  \label{Bcty}%
\end{equation}
where
$$C_{a,c}\  =\ 3 + \frac{\Ctrace^2}{2} \max\left\{ \amin^{-1} , \ \cmax^2 \left(\frac{1}{\omega_0^2} + \betamax^2\right)\right\}.  $$
\item[(ii)] The problems
\eqref{HelmAbs} and \eqref{DualAbs} have unique solutions in $\mathcal{H}$.
\item[(iii)] Under the additional assumption that the
right-hand sides in \eqref{HelmAbs} and \eqref{DualAbs} are given by $F\left(
v\right)  :=\left(  f,v\right)  $ and $G\left(  v\right)  :=\left(
v,\lambda\right)  $ for some functions $f,\lambda\in L^{2}\left(
\Omega\right)  $. Then, there exists a constant $C_{\mathrm{stab}}$
independent of $f$ and $\lambda$ such that the corresponding solutions $u$ and
$z$ satisfy%
\begin{equation}
\Vert u\Vert_{\mathcal{H},a,c}\leq C_{\mathrm{stab}}\left\Vert f\right\Vert
\quad\text{and\quad}\Vert z\Vert_{\mathcal{H},a,c}\leq C_{\mathrm{stab}%
}\left\Vert \lambda\right\Vert . \label{Cstabintro}%
\end{equation}
In general this constant depends on $a$, $c$, $\omega$, and $\Omega$.
\end{itemize} 
\end{theorem}
}
%

\proof
We give the proof for \eqref{HelmAbs}. The proof for \eqref{DualAbs} is
identical. We introduce the parameter-dependent norm on $\mathcal{H}$:
\[
\Vert v\Vert_{\mathcal{H},a,c}^{2}\ :=\ \int_{\Omega}\left\{  a|\nabla
v|^{2}+\left(  \frac{\omega}{c}\right)  ^{2}|u|^{2}\right\}  \ ,
\]
and we begin by writing
\begin{equation}
B_{a,c}\ =\ b_{1}+b_{2}+b_{\Gamma}\ , \label{Bstar}%
\end{equation}
where
\begin{align}
b_{1}\left(  u,v\right)  :=\int_{\Omega}\{a\nabla u.\nabla\overline{v}+\left(
\frac{\omega}{c}\right)  ^{2}u\overline{v}\},  &  \quad b_{2}\left(
u,v\right)  :=-2\int_{\Omega}\left(  \frac{\omega}{c}\right)  ^{2}%
u\overline{v},\\
\text{and}\quad\quad b_{\Gamma}\left(  u,v\right)   &  :=-\operatorname*{i}%
\omega\int_{\Gamma_{N}}\beta u\overline{v}. \label{defeta}%
\end{align}
We observe that $b_{1}$ and $b_{2}$ are Hermitian. Moreover
\begin{equation}
b_{1}\left(  u,u\right)  =\left\Vert u\right\Vert _{\mathcal{H},a,c}^{2},
\label{b1coer}%
\end{equation}
and a combination of H\"{o}lder and Cauchy-Schwarz inequalities leads to the
continuity estimates:
\begin{align}
\left\vert b_{1}\left(  u,v\right)  \right\vert  &  \leq\left\Vert
u\right\Vert _{\mathcal{H},a,c}\left\Vert v\right\Vert _{\mathcal{H}%
,a,c},\nonumber\\
\left\vert b_{2}\left(  u,v\right)  \right\vert  &  \leq2\left\Vert
\frac{\omega}{c}u\right\Vert \left\Vert \frac{\omega}{c}v\right\Vert
\leq2\left\Vert u\right\Vert _{\mathcal{H},a,c}\left\Vert v\right\Vert
_{\mathcal{H},a,c},\label{a1a2a}\\
\left\vert b_{\Gamma}\left(  u,v\right)  \right\vert  &  \leq\left\Vert
\sqrt{\omega|\beta|}u\right\Vert _{\Gamma}\left\Vert \sqrt{\omega|\beta
|}v\right\Vert _{\Gamma}. \label{a1a2b}%
\end{align}
We recall also the multiplicative trace inequality:
\begin{equation}
\Vert u\Vert_{\Gamma}\ \leq\ C_{\mathrm{trace}}\Vert u\Vert^{1/2}\Vert
u\Vert_{H^{1}(\Omega)}^{1/2} \label{eq:mtrace}%
\end{equation}
(For $d=2,3$, this is the last formula in \cite[p.41]{Grisvard85}. For $d=1$
it can be obtained by considering the integral of $(|u|^{2}Z)^{\prime}$ where
$Z$ is the linear function with values $-1,\ 1$ at the left- and right-hand
boundaries of the domain.)
Combining this with Young's inequality we obtain%
\begin{align}
\left\Vert \sqrt{\omega\beta}u\right\Vert _{\Gamma}^{2}  &  \ \leq
\ C_{\operatorname*{trace}}^{2}\omega\beta_{\max}\left\Vert u\right\Vert
\left\Vert u\right\Vert _{H^{1}\left(  \Omega\right)  }\leq
C_{\operatorname*{trace}}^{2}\left(  \frac{\omega^{2}\beta_{\max}^{2}}%
{2}\left\Vert u\right\Vert ^{2}+\frac{1}{2}\left\Vert u\right\Vert
_{H^{1}\left(  \Omega\right)  }^{2}\right) \nonumber\\
&  \ \leq\ C_{\operatorname*{trace}}^{2}\left(  \frac{\left(  1+\omega
^{2}\beta_{\max}^{2}\right)  }{2}\frac{c_{\max}^{2}}{\omega^{2}}\left\Vert
\frac{\omega}{c}u\right\Vert ^{2}+\frac{1}{2a_{\min}}\left\Vert \sqrt{a}\nabla
u\right\Vert ^{2}\right) \nonumber\\
&  \ \leq\ C_{3}^{2}\left\Vert u\right\Vert _{\mathcal{H},a,c}^{2},
\label{Ctrace1}%
\end{align}
with
\[
C_{3}=\frac{C_{\operatorname*{trace}}}{\sqrt{2}}\max\left\{  a_{\min}%
^{-1/2},c_{\max}\sqrt{\frac{1}{\omega_{0}^{2}}+\beta_{\max}^{2}}\right\}  .
\]
This proves the continuity of $b_{\Gamma}$ i.e.
\[
\left\vert b_{\Gamma}\left(  u,v\right)  \right\vert \leq C_{3}^{2}\left\Vert
u\right\Vert _{\mathcal{H},a,c}\left\Vert v\right\Vert _{\mathcal{H},a,c}.
\]
{\igg The result  (i) then follows directly.}  
 {\igg To prove (ii),}  for $u\in\mathcal{H}$, let $K_{2}u$ be the unique solution of
the problem $b_{1}(K_{2}u,v)=b_{2}(u,v),\ v\in\mathcal{H}$, which is
guaranteed to exist by the Lax-Milgram Lemma. Similarly, let $K_{\Gamma}u$
denote the unique solution of $b_{1}(K_{\Gamma}u,v)=b_{\Gamma}(u,v),\ v\in
\mathcal{H}$, and let $SF\in\mathcal{H}$ denote the unique solution of the
problem $b_{1}(SF,v)=F(v),\ v\in\mathcal{H}$, where $F$ appears on the
right-hand side of \eqref{HelmAbs}. Then it is easy to see that
\eqref{HelmAbs} is equivalent to the operator equation in $\mathcal{H}$:
\begin{equation}
(I+K_{2}+K_{\Gamma})u=SF\ . \label{255}%
\end{equation}
Moreover we claim that the operators $K_{2}$ and $K_{\Gamma}$ are compact.
(This is verified at the end of the proof.) Hence by the Fredholm Alternative,
uniqueness for problem \eqref{255} implies unique solvability.

To prove uniqueness, suppose $F=0$ and let $u\in\mathcal{H}$ be a solution of
\[
B_{\igg {a,c}} (u,v)=0,\quad\text{for all}\quad v\in\mathcal{H}.
\]
Putting $v=u$ and taking the imaginary part (and noting $\omega\geq\omega
_{0}>0$), we obtain
\[
\int_{\Gamma_{N}}\beta|u|^{2}=0\,
\]
and then Assumption \ref{ass:UCP} (1) implies that $u=0$ almost everywhere on
$\gamma$. Using Assumption \ref{ass:UCP} (2), (3) we can extend $u$ by zero to
a function $u^{\ast}$ on $\mathcal{H}^{\ast}=H^{1}(\Omega^{\ast})$. Defining
\[
B^{\ast}\left(  u,v\right)  :=\int_{\Omega^{\ast}}\left\{  a^{\ast}\nabla
u \cdot \nabla\overline{v}\ +\ \left(  \frac{\omega}{c^{\ast}}\right)
^{2}u\overline{v}\right\}  \quad\forall u,v\in\mathcal{H}^{\ast},
\]
we have $B^{\ast}(u^{\ast},v)=0$ for all $v\in\mathcal{H}^{\ast}$. Now, since
$u^{\ast}$ vanishes on $\Omega^{\ast}\backslash\Omega$, Theorem \ref{thm:UCP}
tells us that $u$ is identically zero. (The assumptions of Theorem
\ref{thm:UCP} are satisfied because of Assumptions \ref{ass:UCP} (3) and (4).)

To finish the proof we show the compactness of $K_{2},K_{\Gamma}$. By
\eqref{b1coer} and \eqref{a1a2a} we have
\[
\Vert K_{2}u\Vert_{\mathcal{H},a,c}\ \leq\ 2\left\Vert \frac{\omega}%
{c}u\right\Vert \ \leq\ 2\frac{\omega}{c_{\min}}\Vert u\Vert,
\]
which shows $K_{2}$ is bounded as an operator from $L_{2}(\Omega)$ to
$\mathcal{H}$, and is thus compact on $\mathcal{H}$. Similarly, using
\eqref{a1a2b} and \eqref{Ctrace1}, we have
\[
\Vert K_{\Gamma}u\Vert_{\mathcal{H},a,c}\ \leq\ C_{3}\left\Vert \sqrt
{\omega|\beta|}u\right\Vert _{\Gamma}\ \leq C_{3}\sqrt{\omega\beta_{\max}%
}\Vert u\Vert_{\Gamma}\ \leq\ C_{3}C_{\mathrm{trace}}^{\prime}\sqrt
{\omega\beta_{\max}}\Vert u\Vert_{H^{3/4}(\Omega)}\ ,
\]
where we used the continuity of the trace operator from $H^{3/4}%
(\Omega)\rightarrow L^{2}(\Gamma)$, with continuity constant
$C_{\mathrm{trace}}^{\prime}$. Now since $H^{3/4}(\Omega)$ is compactly
embedded in $H^{1}(\Omega)$ we then have compactness of $K_{\Gamma}$.%

{\igg Part (iii) then follows immediately because (for example)  $\Vert SF \Vert_{\cH} \lesssim  \Vert F \Vert_{\cH} \leq \Vert f \Vert$ (where the  hidden constant may depend on $a,c, \omega, \Omega$).  }

\endproof

In
some of our applications in Section \ref{SecGalDisc}, we will employ the
well-posedness of the Helmholtz problem in the following setting.

\begin{corollary}
\label{Cora=1}Let $\Omega\subset\mathbb{R}^{d}$ be a{ bounded Lipschitz}
domain and assume $a\in C^{0,1}\left(  \Omega,\left[  a_{\min},a_{\max
}\right]  \right)  $ and $c\in L^{\infty}\left(  \Omega,\left[  c_{\min
},c_{\max}\right]  \right)  $ for some $0<a_{\min}\leq a_{\max}<\infty$ and
$0<c_{\min}\leq c_{\max}<\infty$. {\igg Let $\Gamma_{N} \subseteq \Gamma$ have positive $d-1$ dimensional measure},  and let
$\beta:=\sqrt{a}/c$. Assume that the right-hand sides in \eqref{HelmAbs} and
\eqref{DualAbs} are given by $F\left(  v\right)  :=\left(  f,v\right)  $ and
$G\left(  v\right)  :=\left(  v,\lambda\right)  $ for some functions
$f,\lambda\in L^{2}\left(  \Omega\right)  $. Then, there exists a constant
$C_{\mathrm{stab}}=C_{\mathrm{stab}}\left(  \omega,a,c,\Omega,d\right)  $
independent of $f$ and $\lambda$ such that the corresponding solutions $u$ and
$z$ satisfy (\ref{Cstabintro}).
\end{corollary}

%

\proof
It is easy to verify that the assumptions in this corollary imply Assumption
\ref{ass:UCP}. In particular, it is well known that the Lipschitz function $a$
can be extended to a Lipschitz function in $\mathbb{R}^{d}$ (with same
Lipschitz constant).%
\endproof



\begin{remark}
The condition on $c^{\ast}$ in Assumption \ref{ass:UCP} could be relaxed,
since Theorem \ref{thm:UCP} only requires that $\kappa=\left(  \omega
/c\right)  ^{2}\in L^{p}\left(  \Omega^{\ast}\right)  $ for some $p>1$.
\end{remark}

\section{Frequency Explicit Estimates for $\Omega\subset\mathbb{R}^{d}$,
$d\geq2$\label{SecDleq2}}


The first frequency explicit stability estimate for the Helmholtz problem with
constant coefficients was given in \cite[Prop 8.1.4]{MelenkDiss}. There the
key idea (used again in many subsequent works) was to use the test function of
Rellich type
\begin{equation}
v:=\mathbf{x}\cdot\nabla u\label{rellich}%
\end{equation}
in the weak form \eqref{eq:weak} and combine the resulting identity with
estimates obtained using the test function $v=u$ and taking real and imaginary
parts. An alternative way of thinking about this is to use a parametrized
linear combination of these test functions (the so-called Morawetz multiplier)
and then to choose the parameters appropriately to obtain the desired result.
This method can also be applied to the heterogeneous Helmholtz problem,
leading to a stability estimate subject to a strong restriction on the
coefficients. Since this is the starting point for our analysis, we provide a
sketch of this procedure for the restricted case {\igg $a = 1 $}  and $c$ variable, with
remarks as to how this can be generalised afterwards.


\begin{assumption}
\label{Agend}\qquad

\begin{enumerate}
\item $\Omega\subset\mathbb{R}^{d}$, $d\geq2$, is a{ Lipschitz} domain which
is star-shaped with respect{ to a ball centred at the origin,} i.e., there
exists a constant $\gamma>0$ such that%
\begin{equation}
\mathbf{x}\cdot \mathbf{n}\geq\gamma\qquad\text{for all }\quad\mathbf{x}\in
\Gamma\label{Astarshape}%
\end{equation}
and we
set
$R:=\sup\left\{  \left\vert \mathbf{x}\right\vert :\mathbf{x}\in
\Omega\right\}  $.

\item We restrict to (\ref{eq:Helm}) with $a=1$ and $\beta=1/c$.

\item $\Gamma=\Gamma_{{N}}$, i.e., we consider the pure
impedance problem, so $\mathcal{H} = H^{1}(\Omega)$.

\item The functions on the right-hand side of (\ref{eq:weak}) satisfy $f\in
L^{2}\left(  \Omega\right)  $ and $g=0$.

\end{enumerate}
\end{assumption}

\noindent 
{\igg {\em Some remarks on Assumption \ref{Agend}}:  We require $\Omega$ to be star-shaped since then
 Theorem \ref{MainTheoremdl2} becomes a generalization  of \cite[Prop. 8.1.4]%
{MelenkDiss}. If the star-shaped  requirement is removed we expect that the proof will become
much more involved,  because the integral equation techniques which are employed
in \cite{EsMe:10} and \cite{Sp:13} to remove the star-shaped
condition for Helmholtz equations with constant coefficients rely on the
explicit knowledge of the fundamental solution (which is not known for
heterogeneous Helmholtz equations). The extension of Theorem
\ref{MainTheoremdl2} below to more general coefficients will be discussed in Remark
\ref{RemCommMainTheo}. Conditions (3) and (4) in Assumption \ref{Agend} can be
generalized in a straightforward way to cover the cases $g\neq0$ and
$\Gamma=\overline{\Gamma_{{N}}}\cup\overline{\Gamma
_{{D}}}$ provided the impedance part $\Gamma_{{N}}$ has positive $d-1$ dimensional surface measure.  We do not do that here to avoid making the paper longer. 
{Under Assumption \ref{Agend}, the equation \eqref{eq:Helm} should be understood
  distributionally and is equivalent to the weak form \eqref{eq:weak}. }
}

\begin{theorem}
\label{MainTheoremdl2}Let Assumption \ref{Agend} be satisfied; let $\omega
\geq\omega_{0}$ for some $\omega_{0}>0$ and let $c\in L^{\infty}\left(
\Omega,\left[  c_{\min},c_{\max}\right]  \right)  $ for some $0<c_{\min}\leq
c_{\max}<\infty$. Further we assume $c\in C^{0,1}\left(  \overline{\Omega
}\right)  $ and that there exists some $\theta>0$ such that%
\begin{equation}
{\color{black}\frac{\mathbf{x}\cdot \nabla c(\mathbf{x})}{c(\mathbf{x})}%
\ \leq\ 1-\theta,\quad\text{for all }\quad\mathbf{x}\in\overline{\Omega}\ .}
\label{eq:condc}%
\end{equation}
Let  $u\in\mathcal{H}$ denote  the
solution of (\ref{eq:weak}). Then  we have the a priori bound:
\begin{equation}
\left\Vert \nabla u\right\Vert +\left\Vert \frac{\omega}{c}u\right\Vert
\ \leq\ \color{black}{\Cstab}\left\Vert f\right\Vert , \label{L2estimate}%
\end{equation}
where $C_{\mathrm{stab}}$ depends continuously on the positive real numbers
$\omega_{0},c_{\min},c_{\max},R,\gamma,{d},\theta$, {\igg   but is independent of $\omega$}. Moreover $C_{\mathrm{stab}%
}$ may become unbounded if one or more of these parameter tends to $0$ or
$\infty$.
\end{theorem}

\begin{remark}\label{rem:Adjoint}

    
\begin{itemize}
\item[(i)] 
  The same estimate holds for the adjoint problem where the
sign of the boundary integral term in the definition of $B_{a,c} $ in
\eqref{eq:weak} is changed from negative to positive.
{\igg \item[(ii)] 
The class of coefficients defined by \eqref{eq:condc} 
includes examples in which   $c$  is  far from constant. 
 If,  for example,  
the domain $\Omega$ is a sphere centred on the origin 
and $c$ is a radial function which decays as we move out from the origin,   then \eqref{eq:condc} is always satisfied, no matter how large $c$ is at the origin or how fast it decays. The recent papers \cite{FeLiLo:14,Wub} contain a more refined convergence analysis  than that given  here, but they  make the  stricter  assumption  that the wave speed is a small perturbation of a constant and, moreover,  the perturbation is required to decay with $\mathcal{O}(\omega^{-1})$ as $\omega \rightarrow \infty$. }
  
\end{itemize}

\end{remark}

%

\proof
First we note that the stated assumptions allow us to apply
Corollary \ref{Cora=1} which implies the existence and uniqueness of the
solution $u$. In the following, the parameters $\varepsilon,\varepsilon
^{\prime}$, $\varepsilon_{j},\varepsilon_{j}^{\prime},\ldots$ denote positive
real numbers, initially arbitrary but eventually fixed. Let $u$ denote the
solution of (\ref{eq:weak}). First, we choose $v=u$ in equation (\ref{eq:weak}%
) and consider the real part of the equation. 
This leads
to%
\[
\left\Vert \nabla u\right\Vert ^{2}\leq\left\Vert \frac{\omega}{c}u\right\Vert
^{2}+\frac{\varepsilon_{1}}{2}\left\Vert u\right\Vert ^{2}+\frac
{1}{2\varepsilon_{1}}\left\Vert f\right\Vert ^{2}.
\]
The choice $\varepsilon_{1}=\varepsilon_{1}^{\prime}\frac{\omega^{2}}{c_{\max
}^{2}}$ yields
\begin{equation}
\left\Vert \nabla u\right\Vert ^{2}\leq\left(  1+\frac{\varepsilon_{1}%
^{\prime}}{2}\right)  \left\Vert \frac{\omega}{c}u\right\Vert ^{2}%
+\frac{c_{\max}^{2}}{2\varepsilon_{1}^{\prime}\omega^{2}}\left\Vert
f\right\Vert ^{2}. \label{gradest1}%
\end{equation}
{Recalling $\beta=1/c$ and $g=0$ and using the imaginary part of equation
(\ref{eq:weak}) with $v=-u$} we deduce%
\[
\left(  \frac{\omega}{c}u,u\right)  _{\Gamma}\leq\frac{1}{2}\left(
\varepsilon_{2}\omega\left\Vert u\right\Vert ^{2}+\frac{1}{\varepsilon
_{2}\omega}\left\Vert f\right\Vert ^{2}\right)  \ ,
\]
from which it follows that
\begin{equation}
\left\Vert \frac{\omega}{c}u\right\Vert _{\Gamma}^{2}\leq\frac{\omega}%
{c_{\min}}\left(  \frac{\omega}{c}u,u\right)  _{\Gamma}\ \leq\ \frac
{1}{2c_{\min}}\left(  \varepsilon_{2}c_{\max}^{2}\left\Vert \frac{\omega}%
{c}u\right\Vert ^{2}+\frac{1}{\varepsilon_{2}}\left\Vert f\right\Vert
^{2}\right)  . \label{uest1}%
\end{equation}


{\color{black} Next, we choose $v$ as in \eqref{rellich}. Then it follows by
elementary vector calculus that
\begin{align}
&  -2\Re\int_{\Omega}\left(  \Delta u+\left(  \frac{\omega}{c}\right)
^{2}u\right)  \overline{v}\nonumber\\
&  \mbox{\hspace{1in}}\ =\ (2-d)\Vert\nabla u\Vert^{2}+\int_{\Omega}%
\nabla\cdot\left(  \left(  \frac{\omega}{c}\right)  ^{2}\mathbf{x}\right)
|u|^{2}\nonumber\\
&  \mbox{\hspace{1.1in}}+\int_{\Gamma}(\mathbf{x}\cdot\mathbf{n})\left(
|\nabla u|^{2}-\left(  \frac{\omega}{c}\right)  ^{2}|u|^{2}\right)
-2\Re\left(  \mathrm{i}\int_{\Gamma}\left(  \frac{\omega}{c}\right)
u\overline{v}\right)  . \label{eq:mult}%
\end{align}
(More precisely \eqref{eq:mult} is first proved for arbitrary $u\in C^{\infty
}(\overline{\Omega})$, and with $v$ as in \eqref{rellich} by elementary vector
calculus.
When $u$ is the actual solution to \eqref{eq:Helm}, \eqref{eq:ImpBC} then
\[
u\in V(\Omega):=\{v\in H^{1}(\Omega):\Delta u\in L^{2}(\Omega),\ \partial
u/\partial n\in L^{2}(\Gamma)\ ,u|_{\Gamma}\in H^{1}(\Gamma)\}.
\]
The proof of \eqref{eq:mult} is completed by observing that $C^{\infty
}(\overline{\Omega})$ is dense in $V(\Omega)$ (see, e.g. \cite{GaGrSp:15} for
an analogous argument).
Then, rearranging \eqref{eq:mult} and recalling \eqref{eq:Helm}, leads to
\begin{align}
&  \int_{\Omega}\nabla\cdot \left(  \left(  \frac{\omega}{c}\right)  ^{2}%
\mathbf{x}\right)  \left\vert u\right\vert ^{2}+\int_{\Gamma}\left(
\mathbf{x}.\mathbf{n}\right)  \left\vert \nabla u\right\vert ^{2}\nonumber\\
&  \mbox{\hspace{0.5in}}=\ \int_{\Gamma}\left(  \mathbf{x}.\mathbf{n}\right)
\left(  \frac{\omega}{c}\right)  ^{2}\left\vert u\right\vert ^{2}%
{\color{black}+}2\Re\left(  \mathrm{i}\left(  \left(  \frac{\omega}{c}\right)
u,v\right)  _{\Gamma}\right)  \ +\ 2\Re\left(  f,v\right)  +\left(
d-2\right)  \left\Vert \nabla u\right\Vert ^{2}\nonumber\\
&  \mbox{\hspace{0.5in}}\leq\ R\left\Vert \frac{\omega}{c}u\right\Vert
_{\Gamma}^{2}\ +\ 2\left\Vert \frac{\omega}{c}u\right\Vert _{\Gamma}\left\Vert
v\right\Vert _{\Gamma}\ +\ \left(  \frac{1}{\varepsilon}\left\Vert
f\right\Vert ^{2}+\varepsilon\left\Vert v\right\Vert ^{2}\right)
+(d-2)\Vert\nabla u\Vert^{2}.\nonumber
\end{align}
Since $\left\Vert v\right\Vert _{\Gamma}\leq R\left\Vert \nabla u\right\Vert
_{\Gamma}$, we obtain%
\begin{align*}
&  \int_{\Omega}\nabla\cdot \left(  \left(  \frac{\omega}{c}\right)  ^{2}%
\mathbf{x}\right)  \left\vert u\right\vert ^{2}+\int_{\Gamma}\left(
\mathbf{x}.\mathbf{n}\right)  \left\vert \nabla u\right\vert ^{2}\\
&  \mbox{\hspace{0.5in}}\leq\ R\left\Vert \frac{\omega}{c}u\right\Vert
_{\Gamma}^{2}+R\left(  \varepsilon^{\prime}\left\Vert \frac{\omega}%
{c}u\right\Vert _{\Gamma}^{2}+\frac{1}{\varepsilon^{\prime}}\left\Vert \nabla
u\right\Vert _{\Gamma}^{2}\right)  \ +\ \left(  \left(  \varepsilon
R^{2}+d-2\right)  \left\Vert \nabla u\right\Vert ^{2}+\frac{1}{\varepsilon
}\left\Vert f\right\Vert ^{2}\right)  .
\end{align*}
Now recall the assumption of star-shapedness (\ref{Astarshape}) and choose
$\varepsilon^{\prime}=2R/\gamma$ to obtain {%
\begin{align}
\int_{\Omega}\nabla\cdot \left(  \left(  \frac{\omega}{c}\right)  ^{2}%
\mathbf{x}\right)  \left\vert u\right\vert ^{2}+\frac{\gamma}{2}\left\Vert
\nabla u\right\Vert _{\Gamma}^{2}  &  \ \leq\ R\left(  1+\frac{2R}{\gamma
}\right)  \left\Vert \frac{\omega}{c}u\right\Vert _{\Gamma}^{2}\nonumber\\
&  \ +\ \left(  \left(  \varepsilon R^{2}+d-2\right)  \left\Vert \nabla
u\right\Vert ^{2}+\frac{1}{\varepsilon}\left\Vert f\right\Vert ^{2}\right)  .
\label{eq:3rdlast}%
\end{align}
}We now employ (\ref{gradest1}), (\ref{uest1}) to estimate the first two terms
on the right-hand side of \eqref{eq:3rdlast}. After a rearrangement this
gives
\begin{align}
&  \int_{\Omega}\nabla\cdot \left(  \left(  \frac{\omega}{c}\right)  ^{2}%
\mathbf{x}\right)  \left\vert u\right\vert ^{2}\ +\ \frac{\gamma}{2}\left\Vert
\nabla u\right\Vert _{\Gamma}^{2}\ \leq\ \delta\left\Vert \frac{\omega}%
{c}u\right\Vert ^{2}\nonumber\\
&  \mbox{\hspace{1in}}+\left(  \frac{R}{2c_{\min}}\left(  1+\frac{2R}{\gamma
}\right)  \frac{1}{\varepsilon_{2}}+\left(  \varepsilon R^{2}+d-2\right)
\frac{c_{\max}^{2}}{2\varepsilon_{1}^{\prime}\omega^{2}}+\frac{1}{\varepsilon
}\right)  \left\Vert f\right\Vert ^{2} \label{2ndlast}%
\end{align}
with%
\begin{equation}
\delta-(d-2)\ =\ \frac{R}{2c_{\min}}\left(  1+\frac{2R}{\gamma}\right)
\varepsilon_{2}c_{\max}^{2}+\varepsilon R^{2}\left(  1+\frac{\varepsilon
_{1}^{\prime}}{2}\right)  +(d-2)\frac{\varepsilon_{1}^{\prime}}{2}.
\label{eq:besmall}%
\end{equation}
Note that, using our assumption \eqref{eq:condc}
\[
\nabla\cdot \left(  \left(  \frac{\omega}{c}\right)  ^{2}\mathbf{x}\right)
=\left(  \frac{\omega}{c}\right)  ^{2}\left(  d-2\frac{\mathbf{x}.\nabla c}%
{c}\right)  \geq((d-2)+2\theta)\left(  \frac{\omega}{c}\right)  ^{2}\ .
\]
}

Hence, by making the right-hand side of \eqref{eq:besmall} small enough we can
\textquotedblleft absorb\textquotedblright\ the term $\delta\left\Vert
\frac{\omega}{c}u\right\Vert ^{2}$ on the right-hand side of (\ref{2ndlast})
into the left-hand side.
We do this by first choosing ${\varepsilon_{1}^{\prime}}=\frac{2\theta}%
{3\max\{d-2,1/2\}}$, which ensures{ $(d-2)\varepsilon_{1}^{\prime}/2\leq
\theta/3$}. Then we choose $\varepsilon,\ \varepsilon_{2}$ so that
\[
\varepsilon R^{2}\left(  1+\frac{\varepsilon_{1}^{\prime}}{2}\right)
\ =\ \frac{\theta}{3}\ =\ \frac{R}{2c_{\min}}\left(  1+\frac{2R}{\gamma
}\right)  \varepsilon_{2}c_{\max}^{2}.
\]
The right hand side of \eqref{eq:besmall} is then bounded{ from above} by
$\theta$, and we have derived the estimate%
\[
{\theta}\left\Vert \frac{\omega}{c}u\right\Vert ^{2}+\frac{\gamma}%
{2}\left\Vert \nabla u\right\Vert _{\Gamma}^{2}\leq\left(  \frac{R}{2c_{\min}%
}\left(  1+\frac{2R}{\gamma}\right)  \frac{1}{\varepsilon_{2}}+\left(
\varepsilon R^{2}+d-2\right)  \frac{c_{\max}^{2}}{2\varepsilon_{1}^{\prime
}\omega^{2}}+\frac{1}{\varepsilon}\right)  \left\Vert f\right\Vert ^{2}.
\]
This leads to the final weighted $L^{2}$ estimate (\ref{L2estimate}). The
estimate of $\left\Vert \nabla u\right\Vert $ follows from this via
(\ref{gradest1}).
\endproof

\begin{remark}
\label{RemCommMainTheo}\quad

\begin{enumerate}
\item[{\igg (i)}] The formulation and proof of Theorem \ref{MainTheoremdl2} for $d=1$
is analogous. We discuss it in some detail for a broader class of coefficients
in \S \ref{Sec1DCase}.

\item[{\igg (ii)}]
Conditions which ensure the stability estimate
\eqref{eq:apriori} (with $C_{\mathrm{stab}}$ independent of $\omega$) for the problem \eqref{eq:Helm}, {\imag with the scalar function 
 $a$ generalised to a positive definite matrix $A$}  may be
written
\begin{equation}
\frac{\mathbf{x}\cdot \nabla c}{c}\ \leq\ \frac{1}{2}-\theta,\quad(\mathbf{x}%
\cdot \nabla)A\ \leq\ A-\theta^{\prime}I,\label{eq:condcA}%
\end{equation}
with $\theta,\theta^{\prime}$ required to be positive. The first condition in
(\ref{eq:condcA}) was introduced in \cite{PerthameVega} for a Helmholtz
equation of the form $\Delta u+n\left(  x\right)  \omega^{2}u=f$ with variable
$n$. By similar multiplier techniques as in \cite{PerthameVega} these results
can be extended to variable $A$ under relatively restrictive conditions on $A$
(see \cite{Brownetal}). In \cite{GrPeSp:17} the condition on $A$ was
formulated in the form (\ref{eq:condcA}). The second inequality should be
understood in the sense of sequilinear forms with the operator $(\mathbf{x}%
\cdot\nabla)$ being applied componentwise to the matrix $A$. In
\cite{GrPeSp:17} it is shown that these conditions imply frequency-independent
stability, not only for the interior impedance problem considered here but
also for Dirichlet scattering problems on infinite and artificially truncated
exterior domains. (Note that when $A$ and $c$ both vary the condition on $c$
in \eqref{eq:condcA} is slightly stronger than that in \eqref{eq:condc}.)

\item[{\igg (iii)}] The main purpose of presenting the proof of Theorem
\ref{MainTheoremdl2} here is to {\igg emphasise}  how far one can get by using the
\textquotedblleft Rellich\textquotedblright\ test function \eqref{rellich}.
The resulting
stability estimates require stronger smoothness requirements on the
coefficients than  those needed for the well-posedness in Theorem \ref{thm:wp}.
Moreover the condition \eqref{eq:condc}, while allowing $c$ to decay arbitrary
quickly in the radial direction, effectively rules out highly oscillatory wave
speeds. This is the starting point for \S \ref{Sec1DCase}, which concerns
stability of problem \eqref{eq:Helm} when \eqref{eq:condc} is not satisfied.
In this case the Rellich test function \eqref{rellich} is not sufficient and
we need to use other \textquotedblleft coefficient dependent\textquotedblright%
\ test functions.


\end{enumerate}
\end{remark}

\section{Finite element error estimates for heterogeneous problems
\label{SecGalDisc}}

{In this section we work under the assumptions as
stated in Corollary \ref{Cora=1}}. {We prove estimates for the \textit{minimal
resolution condition}} and the Galerkin error for {conforming finite element
approximation of \eqref{eq:weak} which are explicit in }$\omega$, $a$, $c$,
$h$, and the stability constant $C_{\mathrm{stab}}${. In general the stability
constant depends also on the coefficients and the wave number. However, the
stronger Assumption \ref{Agend} allows us to apply Theorem
\ref{MainTheoremdl2} and Remark \ref{rem:Adjoint} so that the constant
}$C_{\mathrm{stab}}$ becomes independent of the wavenumber{. }

{Our first results concern the abstract Galerkin method: for a general finite
dimensional subspace $S\subset\mathcal{H}$, 
\begin{equation}
{\igg \text{seek} \quad  u_{S}\in S \quad  \text{such that} \quad}   B_{a,c}\left(  u_{S},v\right)  =F\left(  v\right)  ,\qquad\text{for all}\quad
v\in S. \label{GalDisc}%
\end{equation}
For the homogeneous case ($a=c=1$) existence and
uniqueness of the Galerkin solution  follow by the \textquotedblleft Schatz argument\textquotedblright%
\ (see \cite{Schatz74}); the notion of \textit{adjoint approximability} has
been introduced in \cite{Sauter2005}, \cite{BanjaiSauterI06},
\cite{MelenkSauterMathComp} and has been shown to play a fundamental role in
the theory of \eqref{GalDisc}. Here we generalise this to the heterogeneous
case and refer to \cite{doerfler_sauter} for a similar reasoning in the
context of a posteriori estimates.

{\igg \begin{definition}[Adjoint Approximability]  
Let $T_{a,c}^{\ast}$ denote the solution
operator for the adjoint problem with homogeneous impedance data, that is,  for
${\lambda}${$\in L^{2}(\Omega)$, $z=T_{a,c}^{\ast}$}${\lambda}${ is defined
to be the solution to the adjoint equation
\begin{equation}
B_{a,c}(v,z)=(v,\lambda)\quad\text{for all}\quad v\in H^{1}(\Omega)\ .
\label{adjoint_problem}%
\end{equation}
The well-posedness of this problem is ensured by Corollary \ref{Cora=1}.
Then we define the heterogeneous adjoint approximability constant
$\sigma_{a,c}^{\ast}(S)$ by
\begin{equation}
{\sigma_{a,c}^{\ast}}\left(  S\right)  :=\sup_{\varphi\in L^{2}(\Omega
)\backslash\left\{  0\right\}  }\dfrac{\inf_{v\in S}\left\Vert T_{a,c}^{\ast
}\left(  \left(  \frac{\omega}{c}\right)  ^{2}\varphi\right)  -v\right\Vert
_{\mathcal{H},a,c}}{\left\Vert \frac{\omega}{c}\varphi\right\Vert }.
\label{defhetapproxconst}%
\end{equation}
}

\end{definition} }
Using this we have a result on Galerkin well-posedness and error estimates:

\begin{theorem}
[Discrete stability and convergence]\label{TheoStabDisc}
Suppose the assumptions of Corollary \ref{Cora=1} hold and
suppose
\begin{equation}
\sigma_{a,c}^{\ast}\left(  S\right)  \leq\frac{1}{2C_{a,c}} \ , 
\label{eq:abstract-cond-on-eta}%
\end{equation}
with the continuity constant $C_{a,c}$ as given in \eqref{Bcty}. Then the
discrete problem (\ref{GalDisc}) has a unique solution which satisfies the
error estimates:
\begin{align}
\left\Vert u-u_{S}\right\Vert _{\mathcal{H},a,c}  &  \ \leq\ 2C_{a,c}%
\inf_{v\in S}\left\Vert u-v\right\Vert _{\mathcal{H},a,c}%
,\label{eq:abstract-quasi-optimality-H}\\
\left\Vert \frac{\omega}{c}\left(  u-u_{S}\right)  \right\Vert _{L^{2}%
(\Omega)}  &  \ \leq\ 2C_{a,c}^{2}\sigma_{a,c}^{\ast}\left(  S\right)
\inf_{v\in S}\left\Vert u-v\right\Vert _{\mathcal{H},a,c}.
\label{eq:abstract-quasi-optimality-L2}%
\end{align}

\end{theorem}

%

\proof
We first estimate the $L^{2}$-error in terms of the $H^{1}$-error via the
Aubin-Nitsche technique. Let $e=u-u_{S}$, set $\psi:=T_{a,c}^{\ast}\left(
\left(  \frac{\omega}{c}\right)  ^{2}e\right)  $ and let $\psi_{S}\in S$
denote the best approximation to $\psi$ with respect to $\Vert\cdot
\Vert_{\mathcal{H},a,c}$. {\imag Then, using the definition of $T^\ast$, we have 
$$ \left\Vert \frac{\omega}{c} e \right\Vert^2 = B_{a,c}(e, \psi)$$ and then using}  Galerkin orthogonality and continuity,
we have \begin{align}
\left\Vert \frac{\omega}{c}e\right\Vert ^{2}  &  =B_{a,c}\left(
e,\psi\right)  \leq B_{a,c}\left(  e,\psi-\psi_{S}\right)  \leq C_{a,c}%
\,\left\Vert e\right\Vert _{\mathcal{H},a,c}\left\Vert \psi-\psi
_{S}\right\Vert _{\mathcal{H},a,c}\nonumber\\
&  \leq C_{a,c}\sigma_{a,c}^{\ast}\left(  S\right)  \left\Vert e\right\Vert
_{\mathcal{H},a,c}\left\Vert \frac{\omega}{c}e\right\Vert . \label{l2a3}%
\end{align}
To estimate the $\mathcal{H}$-norm of the error, note that for any $v_{S}\in
S$, we have, again by Galerkin orthogonality (and using (\ref{l2a3})),
\begin{align*}
\left\Vert e\right\Vert _{\mathcal{H},a,c}^{2}  &  =\mathfrak{R}\left(
B_{a,c}\left(  e,e\right)  \right)  +2\left\Vert \frac{\omega}{c}e\right\Vert
^{2}=\mathfrak{R}B_{a,c}\left(  e,u-v_{S}\right)  +2\left\Vert \frac{\omega
}{c}e\right\Vert ^{2}\\
&  \ {\leq}\ C_{a,c}\left\Vert e\right\Vert _{{\mathcal{H},a,c}}\left\Vert
u-v_{S}\right\Vert _{\mathcal{H},a,c}+2\left(  C_{a,c}\,\sigma_{a,c}^{\ast
}\left(  S\right)  \right)  ^{2}\left\Vert e\right\Vert _{\mathcal{H},a,c}%
^{2}.
\end{align*}
Then \eqref{eq:abstract-quasi-optimality-H} follows on application of
(\ref{eq:abstract-cond-on-eta}), and \eqref{eq:abstract-quasi-optimality-L2}
follows by combination of this with \eqref{l2a3}.
\endproof

For practical computations the space $S$ is typically chosen to be an
\textit{hp} finite element space. In this case the role of the
\textquotedblleft resolution condition\textquotedblright%
\ (\ref{eq:abstract-cond-on-eta}) has been studied in detail for Helmholtz
problems with constant coefficients in the sequence of papers
\cite{MelenkSauterMathComp}, \cite{mm_stas_helm2}, \cite{MPS13}. In Theorem
\ref{thm:hetconv} below we give the first extension of this theory to the
heterogeneous case. {\imag To reduce technicalities we restrict the argument to lowest order conforming finite elements.}

 In the argument below we will
make use of the following Poisson problem: given $f\in L^{2}\left(
  \Omega\right)  $ and $g\in H^{1/2}\left(  \Gamma\right)  $,
\begin{equation}
{\igg \text{seek}  \quad u\in
{H^{1}(\Omega)}, \quad \text{
such that} } \quad 
 \left(  \nabla u,\nabla v\right)  +\left(  u,v\right)  =\left(  f,v\right)
+\left(  g,v\right)  _{\Gamma}\qquad\forall v\in{H}^{1}(\Omega).
\label{eq:Poisson}%
\end{equation}


\begin{proposition}
\label{prop:Poisson}Let the assumptions of Corollary \ref{Cora=1} be satisfied
and $\Omega$ be a bounded convex Lipschitz domain. For $g=0$ in
(\ref{eq:Poisson}), the Poisson problem \eqref{eq:Poisson} is $H^{2}$ regular
i.e. there is a constant $C_{\mathrm{reg}}$ such that
\[
\left\Vert u\right\Vert _{H^{2}\left(  \Omega\right)  }\ \leq\ C_{\mathrm{reg}%
}{\left\Vert f\right\Vert }\ .
\]
{\imagg Suppose $d=2$ and let $\Omega$ be a bounded convex polygon.   
For  $g\in H^{1/2}\left(  \Gamma\right)  $ we have} %
\begin{equation}
\left\Vert u\right\Vert _{H^{2}\left(  \Omega\right)  }\ \leq\ C_{\mathrm{reg}%
}\left(  {\left\Vert f\right\Vert +\left\Vert g\right\Vert _{H^{1/2}\left(
\Gamma\right)  }}\right)  \ . \label{H2regularity}%
\end{equation}
\end{proposition}

\begin{proof}
For $g=0$ this is \cite[Theorem 3.2.1.3]{Grisvard85}. For inhomogeneous
Neumann conditions one can use a lifting for the normal trace to transform the
problem to a problem with homogeneous Neumann conditions (see \cite[Lemma
A.1]{mm_stas_helm2} or \cite[Lemma 2.12]{GaGrSp:15} for $d=2$).
\end{proof}

{\igg Throughout the rest of this section we use the following notation. 
 \begin{definition}  
 $\mathcal{T}_{h}$ will denote   a shape-regular family of conforming
simplicial
meshes on $\Omega$ with mesh diameter $h$,   and $S_{h}$
will denote the corresponding space of continuous affine  functions and $I_h$ the usual  nodal interpolation operator .
\end{definition} 
}
We recall that $I_{h}:C\left(  \overline{\Omega}\right)
\rightarrow S_{h}$ is well-defined for functions in $H^{2}\left(
\Omega\right)  $ (for $d=1,2,3$) and satisfies, for some constant
$C_{\mathrm{int}}$,
\begin{equation}
\left\Vert v-I_{h}v\right\Vert +h\left\Vert \nabla\left(  v-I_{h}v\right)
\right\Vert \leq C_{\operatorname*{int}}h^{2}\left\Vert v\right\Vert
_{H^{2}\left(  \Omega\right)  }\qquad\forall v\in H^{2}\left(  \Omega\right)
. \label{eq:est_interp}%
\end{equation}



\begin{theorem}
\label{thm:hetconv}(i) Let the assumptions of Corollary \ref{Cora=1} be
satisfied {\imag and assume in addition that $c \in C^{0,1}(\Omega)$.}
Assume also  that {\imag the solution} of the  Poisson problem (\ref{eq:Poisson}) is $H^{2}$
regular {\imag and satisfies the estimate}  (\ref{H2regularity}). Then,%
\begin{equation}
{\sigma_{a,c}^{\ast}}\left(  S\right)  \leq K\left(  \sqrt{\frac{a_{\max}%
}{a_{\min}}}+\frac{\omega}{\sqrt{a_{\min}}c_{\min}}h\right)  \left(
\frac{c_{\min}}{\omega_{0}}+C_{\mathrm{stab}}\right)  \left(  \frac{\omega
}{\sqrt{a_{\min}}c_{\min}}\right)  ^{2}h \label{sigmastar_est}%
\end{equation}
with%
\begin{equation}
K:=K\left(  a,c,\omega_{0},\Omega\right)  :=C_{\operatorname*{reg}%
}C_{\operatorname*{int}}\sqrt{a_{\min}}\left(  C_{0}+C_{0}^{\prime}{\imag \sqrt{\amin}}  %
+\frac{\kappa_{a}}{\omega_{0}}\sqrt{a_{\min}}c_{\min}\right)  .
\label{defKacOmega}%
\end{equation}
The constants $C_{0}$ and $C_{0}^{\prime}$ are defined in (\ref{def_c0}) and
(\ref{defC_0prime}).

(ii) If, {\imag in addition,} the assumptions of Theorem \ref{MainTheoremdl2} are satisfied,  then
$C_{\mathrm{stab}}$ is independent of $\omega$ and the estimate {\imag \eqref{sigmastar_est} is}   explicit with
respect to the coefficients $c,\omega,h$.

\end{theorem}

%

\proof
For $\varphi\in L^{2}\left(  \Omega\right)  $, let $z:=T_{a,c}^{\ast}\left(
\left(  \frac{\omega}{c}\right)  ^{2}\varphi\right)  $. Then
\eqref{eq:est_interp} leads to%
\begin{equation}
\inf_{v\in S}\left\Vert z-v\right\Vert _{\mathcal{H},a,c}\ \leq\ \left\Vert
z-I_{h}z\right\Vert _{\mathcal{H},a,c}\ \leq\ C_{\operatorname*{int}}h\left(
a_{\max}^{1/2}+\frac{\omega h}{c_{\min}}\right)  \left\Vert z\right\Vert
_{H^{2}\left(  \Omega\right)  }.\label{intest}%
\end{equation}
To get a bound for $\sigma_{a,c}^{\ast}\left(  S\right)  $ it remains to
estimate $\Vert z\Vert_{H^{2}(\Omega)}$ in terms of $\Vert\left(  \frac
{\omega}{c}\right)  \varphi\Vert$.
To do this we write the adjoint Helmholtz equation (\ref{adjoint_problem}) for
$\lambda:=\left(  \frac{\omega}{c}\right)  ^{2}\varphi$ which defines $z$ as {\imag the solution to}   a
Poisson-type problem%
\[%
\begin{array}
[c]{rll}%
-\Delta z+z & =\left(  \frac{\omega}{\sqrt{a}c}\right)  ^{2}\varphi
+\frac{\nabla a}{a}.\nabla z+\left(  1+\left(  \frac{\omega}{\sqrt{a}%
c}\right)  ^{2}\right)  z & \text{in }\Omega\quad\text{a.e.},\\
{\imag \frac{\partial z}{\partial{n}}} &{\imag  =-\mathrm{\operatorname*{i}}\frac{\omega
}{\sqrt{a}c}z} & \text{on }\Gamma.
\end{array}
\]
Then, the $H^{2}$ regularity (\ref{H2regularity}) implies%
\begin{equation}
\Vert z\Vert_{H^{2}(\Omega)}\ \leq\ C_{\mathrm{reg}}\left(  \left\Vert \left(
\frac{\omega}{\sqrt{a}c}\right)  ^{2}\varphi\right\Vert +\frac{\kappa_{a}%
}{\sqrt{a_{\min}}}\left\Vert \sqrt{a}\nabla z\right\Vert +\left\Vert \left(
1+\left(  \frac{\omega}{\sqrt{a}c}\right)  ^{2}\right)  z\right\Vert
+\left\Vert \frac{\omega}{\sqrt{a}c}z\right\Vert _{H^{1/2}(\Gamma)}\right)
,\label{eq:dag1}%
\end{equation}
{\igg where $$\kappa_a := \Vert \nabla a /a \Vert_{\infty}. $$}  
Utilising the pointwise estimate%
\begin{equation}
1+\left(  \frac{\omega}{\sqrt{a}c}\right)  ^{2}\ \leq\ C_{0}\left(
\frac{\omega}{\sqrt{a}c}\right)  ^{2}\quad\text{with}\quad C_{0}=\left(
1+\left(  \frac{a_{\max}^{1/2}c_{\max}}{\omega_{0}}\right)  ^{2}\right)
,\label{def_c0}%
\end{equation}
we obtain%
\begin{equation}
\left\Vert z\right\Vert _{H^{2}\left(  \Omega\right)  }\leq
C_{\operatorname*{reg}}\left(  \frac{\omega}{a_{\min}c_{\min}}\left\Vert
\frac{\omega}{c}\varphi\right\Vert +\frac{\kappa_{a}}{\sqrt{a_{\min}}%
}\left\Vert \sqrt{a}\nabla z\right\Vert +C_{0}\frac{\omega}{a_{\min}c_{\min}%
}\left\Vert \frac{\omega}{c}z\right\Vert +\left\Vert \frac{\omega}{\sqrt{a}%
c}z\right\Vert _{H^{1/2}\left(  \Gamma\right)  }\right)  .\label{eq:z1new}%
\end{equation}
To estimate the last term, we employ a trace inequality and then some
elementary differentiation to obtain%
\begin{equation}
\left\Vert \frac{\omega}{\sqrt{a}c}z\right\Vert _{H^{1/2}\left(
\Gamma\right)  }\ \leq\ C_{\mathrm{trace}}\left\Vert \frac{\omega}{\sqrt{a}%
c}z\right\Vert _{H^{1}\left(  \Omega\right)  }\ \leq\ C_{0}^{\prime}%
\frac{\omega}{\sqrt{a_{\min}}c_{\min}}\left\Vert z\right\Vert _{\mathcal{H}%
,a,c}\ ,\label{eq:z2}%
\end{equation}
\text{with}
\begin{equation}
C_{0}^{\prime}:=C_{\mathrm{trace}}\left(  \frac{1}{\sqrt{a_{\min}}}%
+\frac{c_{\min}}{\omega_{0}}\left(  1+\kappa_{c}+\kappa_{a}/2\right)  \right)
,\label{defC_0prime}%
\end{equation}
{\imag where $\kappa_c = \Vert \nabla c /c \Vert_\infty$}. 
Hence, combining  {\imag \eqref{eq:z2} and \eqref{defC_0prime}  with \eqref{eq:z1new}}, and
using the definition of $C_{\mathrm{stab}}$ as in Corollary \ref{Cora=1} for
the adjoint problem (\ref{adjoint_problem}) with $\lambda=\left(  \frac
{\omega}{c}\right)  ^{2}\varphi$, we obtain
\begin{align}
\left\Vert z\right\Vert _{H^{2}\left(  \Omega\right)  } &  \leq
C_{\operatorname*{reg}}\left(  \frac{\omega}{a_{\min}c_{\min}}\left\Vert
\frac{\omega}{c}\varphi\right\Vert +\left(  C_{0}+C_{0}^{\prime}{\imag \sqrt{\amin}} +\frac
{\kappa_{a}}{\omega_{0}}\sqrt{a_{\min}}c_{\min}\right)  \frac{\omega}{a_{\min
}c_{\min}}\left\Vert z\right\Vert _{\mathcal{H},a,c}\right)  \nonumber\\
                                                        &  \leq C_{\operatorname*{reg}}\left(  1+\left(  C_{0}+C_{0}^{\prime} {\imag \sqrt{\amin}}
+\frac{\kappa_{a}}{\omega_{0}}\sqrt{a_{\min}}c_{\min}\right)  C_{\mathrm{stab}%
}\frac{\omega}{c_{\min}}\right)  \frac{\omega}{a_{\min}c_{\min}}\left\Vert
\frac{\omega}{c}\varphi\right\Vert .\label{estadjz}%
\end{align}
{\imagg The combination of (\ref{intest}) with (\ref{estadjz}) leads to 
\begin{align}
\frac{\inf_{v\in S}\left\Vert z-v\right\Vert _{\mathcal{H},a,c}}{\left\Vert
\frac{\omega}{c}\varphi\right\Vert }  & \ \leq \ h\left(  \sqrt{\frac{a_{\max}%
}{a_{\min}}}+\frac{\omega}{\sqrt{a_{\min}}c_{\min}}h\right)  \times \nonumber \\
& \times C_{\operatorname*{reg}}C_{\operatorname*{int}}\sqrt{a_{\min}}\left[
1+{\left(  C_{0}+C_{0}^{\prime}{\sqrt{a_{\min}}}+\frac{\kappa_{a}}{\omega_{0}%
}\sqrt{a_{\min}}c_{\min}\right)}  C_{\mathrm{stab}}\frac{\omega}{c_{\min}%
}\right]  \frac{\omega}{a_{\min}c_{\min}}  . \label{eq:RHS}\end{align} 
With $K$ as defined in \eqref{defKacOmega}, 
the right hand side of \eqref{eq:RHS}  can be written  
\begin{align*}
& h\left(  \sqrt{\frac{a_{\max}}{a_{\min}}}+\frac{\omega}{\sqrt{a_{\min}%
}c_{\min}}h\right)   \left[  C_{\operatorname*{reg}}C_{\operatorname*{int}}\sqrt{a_{\min}%
}+KC_{\mathrm{stab}}\frac{\omega}{c_{\min}}\right]  \frac{\omega}{a_{\min
}c_{\min}}\\
& =h\left(  \sqrt{\frac{a_{\max}}{a_{\min}}}+\frac{\omega}{\sqrt{a_{\min}%
}c_{\min}}h\right)  \left[  C_{\operatorname*{reg}}C_{\operatorname*{int}}\sqrt{a_{\min}%
}\frac{c_{\min}}{\omega_{0}}+KC_{\mathrm{stab}}\right]  \frac{\omega^{2}%
}{a_{\min}c_{\min}^{2}}. %
\end{align*}
Now,  using%
\[
K\geq C_{\operatorname*{reg}}C_{\operatorname*{int}}\sqrt{a_{\min}}%
\quad (\text{since }C_{0}\geq1), 
\]
we obtain \eqref{sigmastar_est}. 
}
\endproof

\begin{remark}{\igg [{Consequences of  Theorems \ref{MainTheoremdl2}, \ref{TheoStabDisc}, and \ref{thm:hetconv}}]}

\begin{enumerate}
\item[{\igg (i)}] The right-hand side in the estimate (\ref{sigmastar_est}) blows up
if $a_{\min}^{-1}$, $a_{\max}$, $c_{\min}^{-1}$, $c_{\max}$, $\omega_{0}^{-1}%
$, $\kappa_{a}$, $\kappa_{c}$ tend to infinity but remains bounded otherwise.

\item[\igg{(ii)}] The combination of Theorems \ref{TheoStabDisc} and \ref{thm:hetconv}
gives a complete theory for the lowest order
Galerkin discretisation of the
Helmholtz problem with variable $a$ and $c$ in the case when $C_{\mathrm{stab}%
}<\infty$: If $\omega^{2}h$ is chosen sufficiently small (with respect to the
values of $a_{\min}$, $a_{\max}$, $c_{\min},c_{\max}$, $\kappa_{a}$,
$\kappa_{c}$), then the Galerkin method is well-posed and enjoys quasi-optimal
error estimates in the weighted norm $\Vert\cdot\Vert_{\mathcal{H},a,c}$. The
condition on $\omega^{2}h$ becomes more stringent if $C_{\mathrm{stab}}$,
$a_{\min}^{-1}$, $a_{\max}$, $c_{\min}^{-1}$, $c_{\max}$, $\kappa_{a}$ or
$\kappa_{c}$ increase. This result shows the key role played by the stability
constant $C_{\mathrm{stab}}$ in the Galerkin theory. We also know from
\S \ref{SecDleq2} that a sufficient condition for a frequency-independent
bound $C_{\mathrm{stab}}<\infty$ (under the assumptions of Theorem
\ref{MainTheoremdl2}) involves upper bounds on $c_{\min}^{-1}$, $c_{\max}$ and
$\mathbf{x}.\nabla c/c$ and so the stability and discretisation theories are
intimately linked.

{\igg The
resolution condition `$\omega^{2}h$ sufficiently small' for discrete stability proved here is a first
result for the general class of coefficients which are considered in this
paper and shows the importance of $\omega$-explicit estimates of the stability
constant $C_{\operatorname*{stab}}$. For constant coefficients and
coefficients which are a very small perturbation of a   constant function, the resolution
condition for low order finite elements can be relaxed (cf. \cite{Wua,Wub, Wuc}) to
`$\omega^{3}h^{2}$ sufficiently small' while for finite element methods of higher
polynomial order $p$ the condition `$\frac{\omega h}{p}$  sufficiently small' is optimal,
however, subject to the side constraint $p\gtrsim\log\omega$; see
\cite{MelenkSauterMathComp}, \cite{mm_stas_helm2}. These improved results are
based on the \textquotedblleft decomposition lemma\textquotedblright\ in
\cite{MelenkSauterMathComp} and we are not aware of a generalization of this
lemma to the more general  class of heterogeneous  coefficients considered here.}

\item[{\igg (iii)}] For constant coefficients $a=c=1$ it is shown in
\cite{MelenkSauterMathComp}, \cite{mm_stas_helm2} that higher order methods
perform much better in the pre-asymptotic range than low order methods: the
condition `$\omega^{2}h$ sufficiently small'
can then be replaced by `$\omega h/p$ sufficiently small',  
if the polynomial degree is chosen according to {\igg $p\gtrsim\log \omega $}. For
heterogeneous Helmholtz problems this question is open; first numerical
results are reported in \cite{Torres18}.

\end{enumerate}
\end{remark}




\section{The 1-dimensional Case\label{Sec1DCase}}

{\igg In this section we investigate in detail in the case of $1-D$ problems how the stability estimate depends 
on $a$ and $c$, especially in the case   when the coefficients can become oscillatory. In Theorem \ref{MainStabW11} we prove a bound which shows that the stability constant may grow exponentially in the variance of either $a$ or $c$ or both. Then we give an example which shows this estimate to be essentially sharp.} 
     
Without loss of generality we assume $\Omega=\left(  -L,L\right)  $; a problem
on any other interval can be transferred to $\Omega$ via an affine change of
variable. The boundary consists of the two endpoints $\Gamma=\left\{
-L,L\right\}  $ and we consider the following choices of the Dirichlet- and
impedance parts of the boundary:%
\begin{align}
\begin{array}
[c]{rll}%
\Gamma_{{D}}=\emptyset & \text{and }\Gamma_{{N}%
}=\Gamma & \text{\emph{pure impedance,}}\\
\Gamma_{{D}}=\left\{  L\right\}  & \text{and }\Gamma
_{{N}}=\left\{  -L\right\}  & \text{\emph{Dirichlet-impedance,}%
}\\
\Gamma_{{D}}=\left\{  -L\right\}  & \text{and }\Gamma
_{{N}}=\left\{  L\right\}  & \text{\emph{impedance-Dirichlet}.}%
\end{array}
\label{BCprobs} 
\end{align} 
To reduce technicalities we assume throughout this section that $\beta$ given
in \eqref{eq:ImpBC} is
\begin{equation}
\beta\left(  x\right)  ={\frac{\sqrt{a(x)}}{c(x)}}\quad\quad\text{for} \quad
x\in\Gamma_{{N}}. \label{abeta}%
\end{equation}
The weak form is defined using the sesquilinear form on $\mathcal{H}$:
\begin{equation}
B_{a,c}(u,v)\ :=\ \int_{-L}^{L}\left(  au^{\prime}\overline{v}^{\prime
}-\left(  \frac{\omega}{c}\right)  ^{2}u\overline{v}\right)  -\mathrm{i}%
\sum_{x\in\Gamma_{{N}}}\omega\frac{\sqrt{a\left(  x\right)  }%
}{c\left(  x\right)  }u(x)\overline{v}(x). \label{eq:1DSqf}%
\end{equation}
{We recall that the norm is%
\[
\Vert v\Vert_{\mathcal{H},a,c}^{2}\ =\ \Vert\sqrt{a}v^{\prime}\Vert^{2}
+\left\Vert \left(  \frac{\omega}{c}\right)  v\right\Vert ^{2}.
\]
}
Then the problem is: 
\begin{equation}
{\igg    \text{seek} \quad  u\in\mathcal{H}, \quad \text{such that} \quad }
B_{a,c}(u,v)\ =\ F(v)+G(v), \ \quad\text{for all }\quad v\in\mathcal{H},
\label{eq:1DB}%
\end{equation}%
\begin{equation}
\text{where}\quad G(v)=\sum_{x\in\Gamma_{{N}}}g_{x}\overline
{v}(x)\quad\text{and}\quad F(v)=\int_{-L}^{L}f\overline{v}, \label{eq:1DBrhs}%
\end{equation}
for given data $f\in L^{2}(\Omega)$, and $g_{x}\in\mathbb{C}$, $x\in
\Gamma_{{N}}$. {\igg (Note that here the function $g$ in \eqref{eq:ImpBC}  consists of two values $g_{-L}$ and $g_L$.)}

To describe the properties of $a,c$ we need the following definition of functions with a finite number of jumps and changes of sign.
{\imag \begin{definition}
\label{DefSpaces}
We denote by $C_{\operatorname*{pw}}^{1}[-L,L]$  the space of functions
$g:[-L,L]\rightarrow\mathbb{R}$ such that there exists a {\imag finite}  partition
\begin{equation}
-L=  z_{0}<\ldots < z_{N}=L\ ,\label{partition}%
\end{equation}
\text{with } 
\ $ g \in C^{1}\left[
z_{j-1}, z_j \right]$
for each   $ j = 1, \ldots , N$, and 
\begin{align}
\text{either}\quad g^{\prime}(x)>0\quad &  \text{or}\quad g^{\prime}%
(x)\leq  0 , \quad \text{when} \quad  x \in (z_{j-1}, z_j) .  \label{eq:onesigng}
\end{align}
 The partition depends on $g$ and is not unique; once a partition
for $g$ is identified, any refinement of it is also a partition. 
For each $z_j$, we define the one-sided limits%
\[
g^{-}\left(  z_j\right)  :=\lim_{\substack{x\rightarrow z_j\\x< z_j }}g\left(
x\right)  \quad\text{and\quad}g^{+}\left(  z_j\right)  =\lim
_{\substack{x\rightarrow z_j\\x>z_j}}g\left(  x\right)  .
\]
and the jumps of $g$ at each  $z_{j}$ (here
taken from left to right) are defined by 
\[
\left[  g\right]  _{z_{j}}:=\left\{
\begin{array}
[c]{ll}%
g^{-}\left(  z_{j}\right)  -g^{+}\left(  z_{j}\right)   & 1\leq
j \leq N-1,\\
-g^{+}\left(  -L\right)   & j=0,\\
g^{-}\left(  L\right)   & j=N.
\end{array}
\right.
\]
(When {\igg $g'$} is one-signed  and without any discontinuities, we have $z_0 = -L, z_1 = L$ and no interior points in the partition.) 
Then we define the \emph{regular part} of the derivative of
$g\in${$C_{\mathrm{pw}}^{1}\left(  \Omega\right)  $} by%
\[
 \partial_{\operatorname*{pw}}g (x) \ =
\   g^{\prime}(x), \quad x \in (z_{j-1}, z_j), \quad j = 1, \ldots, N ,  
\]
and the  \emph{variation} of $g$ on $[-L,L]$ is  defined by
\[
\mathrm{Var}(g)={\igg \sum_{\ell=1}^{N-1}}|\left[  g\right]  _{z_{\ell}}%
|\ +\ \int_{-L}^{L}|(\partial_{\mathrm{pw}}g)(s)|\mathrm{d}s . 
\]
For   later notational convenience,  we denote the subintervals of \eqref{partition} as: 
\begin{equation}
\tau_{j}=\left(  z_{j-1},z_{j}\right)  ,\quad \jsub .\label{eq:intervals}%
\end{equation}

\end{definition}

In the following we shall make the following assumption on the coefficients $a,c$. 
\begin{assumption} \label{ass:coeffs} 
We assume that $a,c\in C_{\mathrm{pw}}^{1}$ and that 
\[
a_{\min}\leq a(x)\leq a_{\max}\quad\text{and}\quad c_{\min}\leq c(x)\leq
c_{\max},\quad x\in\lbrack-L,L],  
\]
for some positive $\amin, \cmin$. 
 Then, without loss of generality, 
there is a partition (which we again  write as  \eqref{partition}) so that,   
that for each $\tau_{j}, \ \jsub$, 
$a'$ and $c'$ are both one-signed.  

\end{assumption} 
}
 

While the problem is properly defined by (\ref{eq:1DB}), we also wish to
derive estimates using test functions with discontinuities at the points
$z_{j}$, To allow this we rewrite (\ref{eq:1DB}) as an interface problem. By
adapting the argument in \cite[Theorem 1]{FreletDiss} one can show that
(\ref{eq:1DB}) is equivalent to the problem
\begin{equation}
-(au^{\prime})^{\prime}\ -\ \left(  \frac{\omega}{c^{2}}\right)  ^{2}%
u=f\quad\text{in}\quad\tau_{j},\quad\text{for all}\quad{\jsub } \label{eq:localstrong}%
\end{equation}
together with the interface conditions {\imag at interior points}
\begin{equation}
\left[  u\right]  _{z_{j}}=0\quad\text{and}\quad\left[  au^{\prime}\right]
_{z_{j}}=0,\quad \jint  \label{eq:interface}%
\end{equation}
and the boundary conditions (cf. \eqref{eq:ImpBC} with \eqref{abeta})
\begin{equation}
\left(  a\frac{\partial u}{\partial n}-{\mathrm{i}\omega\frac{\sqrt{a}}{c}%
}u\right)  (x)=g_{x}\quad\text{for all} \  x\in\Gamma_{{N}}\quad
\text{and\quad}u\left(  x\right)  =0\quad\text{for }x\in\Gamma
_{\operatorname*{D}} \label{eq:1DstrongBC}%
\end{equation}
with the \textquotedblleft normal\textquotedblright\ derivative defined as
${\partial u}/{\partial n}\left(  \pm L\right)  :=\pm u^{\prime}\left(  \pm
L\right)  .$ {\imag Note that Assumption \ref{ass:coeffs}   allows  $a$ to be discontinuous at
some of the $z_{j}$ (but does  not imply discontinuity at any particular $z_{j}%
$). If $a$ is continuous at $z_{j}$, the  second  interface condition in 
\eqref{eq:interface} simplifies to $[u^{\prime}]_{z_{j}} = 0$}.

\subsection{Stability estimate for oscillatory and jumping coefficients}
\label{subsec:stab} 
{We note that Theorem \ref{thm:wp} ensures that the problem \eqref{eq:1DB} has
a unique solution.}

\begin{lemma}
\label{lem:qubar1Ddiscts} {\imag Suppose Assumption \eqref{ass:coeffs} is satisfied} and 
let $u$ solve \eqref{eq:1DB}. Then for any
{real-valued} $q$ which satisfies $\left.  q\right\vert _{\tau_{j}}\in
C^{1}(\tau_{j})$, for each {\imag $\jsub $} and
$q\left(  x\right)  =0$ for $x\in\Gamma_{{D}}$, we have
\begin{align}
&  \ \frac{1}{2}\int_{-L}^{L}\left(  \partial_{\operatorname*{pw}}\left(
\frac{q}{a} \right)  \left\vert a u^{\prime}\right\vert ^{2}+\omega
^{2}\partial_{\operatorname*{pw}}\left(  \frac{q}{c^{2}}\right)  \left\vert
u\right\vert ^{2} \right) \nonumber\\
&  \mbox{\hspace{0.5in}}-\frac{1}{2}\sum_{j=1}^{N-1}\left(  \left[
\frac{q}{a}\right]  _{z_{j}}|(au^{\prime})(z_{j})|^{2}+\omega^{2}\left[
\frac{q}{c^{2}}\right]  _{z_{j}}\left\vert u\left(  z_{j}\right)  \right\vert
^{2}\right) \nonumber\\
&  \mbox{\hspace{0.5in}}\ \leq\ \frac{3}{2}\sum_{x\in\Gamma_{{N}%
}}\left\vert q\left(  x\right)  \right\vert \left\vert \frac{\omega}{c\left(
x\right)  }u\left(  x\right)  \right\vert ^{2}+\frac{1}{a_{\min}}\sum
_{x\in\Gamma_{{N}}}\left\vert q\left(  x\right)  \right\vert
\left\vert g_{x}\right\vert ^{2}+\left\vert \left(  f,qu^{\prime}\right)
\right\vert . \label{1DRellich}%
\end{align}

\end{lemma}

%

\proof
Suppose $\left.  q\right\vert _{\tau_{j}}\in C^{1}(\tau_{j})$ for each
{\imag $\jsub$}. Then,
\begin{equation}
-\Re\int_{z_{j-1}}^{z_{j}}(au^{\prime})^{\prime}q\bar{u}^{\prime}=-\frac{1}%
{2}\int_{z_{j-1}}^{z_{j}}\frac{q}{a}\left(  \left\vert au^{\prime}\right\vert
^{2}\right)  ^{\prime}=\ \frac{1}{2}\int_{z_{j-1}}^{z_{j}}\left(  \frac{q}%
{a}\right)  ^{\prime}\left\vert au^{\prime}\right\vert ^{2}-\frac{1}{2}\left.
\left(  \frac{q}{a}\left\vert au^{\prime}\right\vert ^{2}\right)  \right\vert
_{z_{j-1}}^{z_{j}}\label{eq:local1}%
\end{equation}
and
\begin{equation}
-\Re\int_{z_{j-1}}^{z_{j}}u\left(  \frac{q}{c^{2}}\right)  \bar{u}^{\prime
}=-\frac{1}{2}\int_{z_{j-1}}^{z_{j}}\left(  \left\vert u\right\vert
^{2}\right)  ^{\prime}\left(  \frac{q}{c^{2}}\right)  =\ \frac{1}{2}%
\ \int_{z_{j-1}}^{z_{j}}\left(  \frac{q}{c^{2}}\right)  ^{\prime}\left\vert
u\right\vert ^{2}-\frac{1}{2}\left.  \left(  \frac{q}{c^{2}}\left\vert
u\right\vert ^{2}\right)  \right\vert _{z_{j-1}}^{z_{j}}.\label{eq:local2}%
\end{equation}
Now add (\ref{eq:local1}) to (\ref{eq:local2}) (multiplied by $\omega^{2}$)
and sum over {\imag $\jsub $,}   to obtain (recalling that $u$ satisfies the interface
conditions (\ref{eq:interface})) 
%
\begin{align}
&  -\Re\int_{-L}^{L}\left(  \partial_{\operatorname*{pw}}(au^{\prime})+\left(
\frac{\omega}{c}\right)  ^{2}u\right)  q\bar{u}^{\prime}=\frac{1}{2}
\int_{-L}^{L} \left(\partial_{\operatorname*{pw}}\left(  \frac{q}{a}\right)
\left\vert au^{\prime}\right\vert ^{2}+\omega^{2}\left(  \partial
_{\operatorname*{pw}}\left(  \frac{q}{c^{2}}\right)  \right)  \left\vert
u\right\vert ^{2}\right)  \label{pw1}\\
&  \qquad\quad-\frac{1}{2}\sum_{j=1}^{N}\left.  \left(  \frac{q}%
{a}\left\vert au^{\prime}\right\vert ^{2}\right)  \right\vert _{z_{j-1}%
}^{z_{j}}-\frac{1}{2}\sum_{j=1}^{N}\omega^{2}\left(  \frac{q}{c^{2}}\left.
\left\vert u\right\vert ^{2}\right)  \right\vert _{z_{j-1}}^{z_{j}}\nonumber\\
&  \qquad=\frac{1}{2}\int_{-L}^{L}\left(  \partial_{\operatorname*{pw}%
}\left({\imag \frac{q}{a}}\right)\left\vert au^{\prime}\right\vert ^{2}+\omega^{2}\left(  \partial
_{\operatorname*{pw}}\left(  \frac{q}{c^{2}}\right)  \right)  \left\vert
u\right\vert ^{2}\right)  \nonumber\\
&  \qquad\quad-\frac{1}{2}\sum_{j=1}^{N-1}\left(  \left[  {\frac{q}{a}%
}\right]  _{z_{j}}\left\vert (au^{\prime})\left(  z_{j}\right)  \right\vert
^{2}+\omega^{2}\left[  \frac{q}{c^{2}}\right]  _{z_{j}}\left\vert u\left(
z_{j}\right)  \right\vert ^{2}\right)  \nonumber\\
&  \qquad\quad-\frac{1}{2}\left.  \left[  \left(  qa\left\vert u^{\prime
}\right\vert ^{2}\right)  +\omega^{2}\left(  \frac{q}{c^{2}}\left\vert
u\right\vert ^{2}\right)  \right]  \right\vert _{-L}^{L}.\nonumber
\end{align}
On the other hand, the left-hand side in (\ref{pw1}) equals $\mathfrak{R}%
\left(  f,qu^{\prime}\right)  $ so that
\begin{align}
&  \frac{1}{2}\int_{-L}^{L}\left(  \partial_{\operatorname*{pw}}\left({\imag \frac{q}{a}} \right)\left\vert u^{\prime}\right\vert
^{2} + \omega^{2}\left(  \partial_{\operatorname*{pw}}\left(  \frac{q}{c^{2}}\right)
\right) \left\vert u\right\vert
^{2} \right)  \nonumber\\
&  -\frac{1}{2}{\igg \sum_{j=1}^{N-1}}\left(  \left[  {\frac{q}{a}}\right]
_{z_{j}}\left\vert (au^{\prime})\left(  z_{j}\right)  \right\vert ^{2}%
+\omega^{2}\left[  \frac{q}{c^{2}}\right]  _{z_{j}}\left\vert u\left(
z_{j}\right)  \right\vert ^{2}\right)  \nonumber\\
&  =\frac{1}{2}\left.  \left[  \left(  a\left\vert u^{\prime}\right\vert
^{2}+\frac{\omega^{2}}{c^{2}}\left\vert u\right\vert ^{2}\right)  q\right]
\right\vert _{-L}^{L}+\mathfrak{R}\left(  f,qu^{\prime}\right)
\ .\label{eq:estfirst}%
\end{align}

The required result is then obtained by estimating the first term on the
right-hand side of \eqref{eq:estfirst}. To do this we use the boundary
conditions (\ref{eq:1DstrongBC}) to obtain, for $x\in\Gamma_{{N}%
}$,
\[
a\left(  x\right)  \left\vert u^{\prime}\left(  x\right)  \right\vert ^{2}%
\leq2\left(  \frac{\omega}{c\left(  x\right)  }\right)  ^{2}\left\vert
u\left(  x\right)  \right\vert ^{2}+{2\frac{\left\vert g_{x}\right\vert ^{2}%
}{a\left(  x\right)  }}.
\]
Combining this with \eqref{eq:estfirst}
yields the result.
\endproof

This lemma leads to the following theorem, which identifies suitable
properties of $q$ which will lead to an \emph{a priori} bound for $u$.
Following this we will describe how to construct $q$ satisfying these properties.

\begin{theorem}
\label{th:stability1D} {\imag Suppose Assumption \eqref{ass:coeffs} is satisfied.} Let $u$ solve \eqref{eq:1DB} and suppose $q$ can be
chosen as in Lemma \ref{lem:qubar1Ddiscts}, with the two additional properties:

\begin{enumerate}
\item For any ${\imag \jsub}$ ,
\begin{equation}
\partial_{\mathrm{pw}}\left(  \frac{q}{a}\right)  (x)\ \geq\ \frac{1}%
{a(x)},\quad\text{and}\quad\partial_{\mathrm{pw}}\left(  \frac{q}{c^{2}%
}\right)  (x)\ \geq\ \frac{1}{c^{2}(x)},\quad x\in\tau_{j}.  \label{pwsder}%
\end{equation}

\item We have the negative {\imag interior} jumps:
\begin{align}
\quad\left[  \frac{q}{a}\right]  _{z_{j}} \leq0  &  \quad\text{and}
\quad\left[  \frac{q}{c^{2}}\right]  _{z_{j}} \leq0, \quad \jint. \label{negjumps}%
\end{align}
\end{enumerate}
Then for $\omega_{0}>$0, $f\in L^{2}\left(  \Omega\right)  $ and
$g:\Gamma_{{N}}\rightarrow\mathbb{C}$ the a priori bound
\begin{equation}
\left\Vert u\right\Vert _{\mathcal{H},a,c}\ \leq\ C_{\operatorname*{stab}%
}^{\operatorname{I}}Q\left\Vert f\right\Vert +C_{\operatorname*{stab}%
}^{\operatorname{II}}\sqrt{Q}\left\Vert g\right\Vert _{\Gamma
_{\operatorname*{N}}} \label{MainStab1D}%
\end{equation}
holds, for all $\omega\geq\omega_{0}$, with $Q=\Vert q\Vert_{L^{\infty}\left(
[-L,L]\right)  }$ , 
\begin{equation}
C_{\operatorname*{stab}}^{\operatorname{I}}:=\frac{2}{\sqrt{a_{\min}}}\left(
1+3\frac{c_{\max}}{c_{\min}}\right)  \quad\text{and\quad}%
C_{\operatorname*{stab}}^{\operatorname{II}}:=\frac{2}{\sqrt{a_{\min}}}%
\sqrt{\frac{3}{2}\frac{c_{\max}}{c_{\min}}+1}. \label{DecCstabI-II}%
\end{equation}

.
\end{theorem}

%

\proof
In the following, $\varepsilon,\varepsilon_{1},\varepsilon_{2}\ldots$ denote
positive real numbers. Also, we will make frequent use of the following
elementary estimate. For two functions $\mu,\nu\in L^{2}\left(  \Omega\right)
$ and any positive function $\delta\in L^{\infty}\left(  \Omega,\left[
\delta_{0},\delta_{1}\right]  \right)  $ for some $0<\delta_{0}\leq\delta
_{1}<\infty$, it follows%
\begin{equation}
\left\vert \left(  \mu,\nu\right)  \right\vert \leq\frac{1}{2}\left\Vert
\delta\mu\right\Vert ^{2}+\frac{1}{2}\left\Vert \frac{\nu}{\delta}\right\Vert
^{2}. \label{pq}%
\end{equation}

Using Lemma \ref{lem:qubar1Ddiscts}, and making use of \eqref{pwsder} and
\eqref{negjumps}, we obtain 
{\imag \begin{align} \frac{1}{2} \Vert u \Vert_{\cH, a,c}^2 \ \leq \ \frac{3}{2} \sum_{x \in \Gamma_N} \vert q(x) \vert \left\vert \frac{\omega}{c(x)} u(x) \right\vert^2 + \frac{1}{\amin} \sum_{x \in \Gamma_N} \vert q(x) \vert \vert g_x \vert^2 + \vert (f, qu')\vert. \label{eq:firstQ} \end{align} 
For convenience we introduce the notation (for suitable functions $f,g$), 
$$ (f,g)_{\Gamma_N} \ = \ \sum_{x \in \Gamma_N} f(x) \overline{g(x)} \quad \text{and} \quad \Vert 
f\Vert_{\Gamma_N}^2 = (f,f)_{\Gamma_N}\ . $$  }
Then \eqref{eq:firstQ} yields
\begin{align}\label{eq:secondQ} 
\ \frac{1}{2Q}\Vert u\Vert_{\mathcal{H},a,c}^{2}\ \leq\frac{3}{2}\left\Vert
\frac{\omega}{c}u\right\Vert _{\Gamma_{{N}}}^{2}+\int_{-L}%
^{L}\left\vert f\right\vert \left\vert u^{\prime}\right\vert \ +\ \frac{1}{a_{\min}}
{\left\Vert g\right\Vert _{\Gamma_{{N}}}^{2}}.
\end{align}
To estimate the first term on the right-hand side of {\imag \eqref{eq:secondQ}}, we insert
$v=u$ into (\ref{eq:1DB}) and take the imaginary part of each side to obtain
\begin{equation}
\left\Vert \sqrt{\frac{\omega\sqrt{a}}{c}}u\right\Vert _{\Gamma
_{\operatorname*{N}}}^{2}\leq\left\vert \left(  f,u\right)  \right\vert
+\left\vert \left(  g,u\right)  _{\Gamma_{{N}}}\right\vert .
\label{estomegacu}%
\end{equation}
We then use (\ref{pq}) with $\mu =\left.  u\right\vert _{\Gamma
_{\operatorname*{N}}}$, $\nu=g$, and $\delta=\sqrt{\frac{\omega\sqrt{a}}{c}}$ to
obtain%
\[
\left\vert \left(  g,u\right)  _{\Gamma_{{N}}}\right\vert
\leq\frac{1}{2}\left\Vert \sqrt{\frac{\omega\sqrt{a}}{c}}u\right\Vert
_{\Gamma_{{N}}}^{2}+\frac{1}{2}\left\Vert \sqrt{\frac{c}%
{\omega\sqrt{a}}}g\right\Vert _{\Gamma_{N}}^{2}%
\]
{so that \eqref{estomegacu} yields}%
\[
\left\Vert \sqrt{\frac{\omega\sqrt{a}}{c}}u\right\Vert _{\Gamma
_{\operatorname*{N}}}^{2}\leq\left\Vert \frac{f}{\delta_{2}}\right\Vert
^{2}+\left\Vert \delta_{2}u\right\Vert ^{2}+\left\Vert \sqrt{\frac{c}%
{\omega\sqrt{a}}}g\right\Vert _{\Gamma_{{N}}}^{2},\quad
{\text{for any}\quad\delta_{2}>0}\,.
\]
{Hence}
\begin{align*}
\left\Vert \frac{\omega}{c}u\right\Vert _{\Gamma_{{N}}}^{2}  &
\leq\frac{\omega}{c_{\min}\sqrt{a_{\min}}}\left\Vert \sqrt{\frac{\omega
\sqrt{a}}{c}}u\right\Vert _{\Gamma_{{N}}}^{2}\\
&  \leq\frac{1}{c_{\min}\sqrt{a_{\min}}}\left(  \omega\left\Vert \frac
{f}{\delta_{2}}\right\Vert ^{2}+c_{\max}\left\Vert \sqrt{\frac{\omega}{c}%
}\delta_{2}u\right\Vert ^{2}+{\frac{c_{\max}}{\sqrt{a_{\min}}}}\left\Vert
g\right\Vert _{\Gamma_{{N}}}^{2}\right)  .
\end{align*}
We choose $\delta_{2}=\sqrt{\varepsilon_{1}\omega/c}$ to finally obtain%
\[
\left\Vert \frac{\omega}{c}u\right\Vert _{\Gamma_{{N}}}^{2}%
\leq\frac{c_{\max}}{c_{\min}\sqrt{a_{\min}}}\left(  \frac{1}{\varepsilon_{1}%
}\left\Vert f\right\Vert ^{2}+\varepsilon_{1}\left\Vert \frac{\omega}%
{c}u\right\Vert ^{2}+{\frac{1}{\sqrt{a_{\min}}}}\left\Vert g\right\Vert
_{\Gamma_{{N}}}^{2}\right)  .
\]
Now, substituting this for the first term on the right-hand side of
{\imag (\ref{eq:secondQ})}  and estimating the second term similarly, we obtain
\begin{align*}
&  \ \frac{1}{2Q}\Vert u\Vert_{\mathcal{H},a,c}^{2}\ \leq\ \frac{3}{2}%
\frac{c_{\max}}{c_{\min}\sqrt{a_{\min}}}\left(  \frac{1}{\varepsilon_{1}%
}\left\Vert f\right\Vert ^{2}+\varepsilon_{1}\left\Vert \frac{\omega}%
{c}u\right\Vert ^{2}+{\frac{1}{\sqrt{a_{\min}}}}\left\Vert g\right\Vert
_{\Gamma_{{N}}}^{2}\right) \\
&  \mbox
{\hspace{1.2in}}+\frac{1}{2\varepsilon_{2}}\left\Vert f\right\Vert ^{2}%
+{\frac{\varepsilon_{2}}{2{a_{\min}}}\left\Vert \sqrt{a}u^{\prime}\right\Vert
^{2}}\ +\ \frac{\left\Vert g\right\Vert _{\Gamma_{{N}}}^{2}%
}{a_{\min}}.
\end{align*}
The choices of $\varepsilon_1$ and $\varepsilon_2$ given by 
\[
\varepsilon_{1}\frac{3}{2}\frac{c_{\max}}{c_{\min}\sqrt{a_{\min}}}=\frac
{1}{4Q}\quad\text{and }\quad{\frac{\varepsilon_{2}}{2a_{\min}}=\frac{1}{4Q}}%
\]
lead to%
\[
\frac{1}{4Q}\Vert u\Vert_{\mathcal{H},a,c}^{2}\ \leq\ Q\left(  {\frac
{1}{a_{\min}}}+9\frac{c_{\max}^{2}}{c_{\min}^{2}a_{\min}}\right)  \left\Vert
f\right\Vert ^{2}+{\frac{1}{a_{\min}}\left(  1+\frac{3}{2}\frac{c_{\max}%
}{c_{\min}}\right)  \left\Vert g\right\Vert _{\Gamma_{{N}}}^{2}} , %
\]
which yields the result after straightforward algebraic manipulations.%
\endproof

{\imag Recall that the coefficients $a,c$ are required to satisfy Assumption \ref{ass:coeffs}.}
 In order to construct an appropriate function $q$ in Theorem
\ref{th:stability1D}, we introduce the following definition. From the function
$a$ defined above, and for each $j$, we define, for $x\in\tau_{j} = (z_{j-1}, z_j)$,
\begin{align}
\widetilde{a}(x)\ =\ \left\{
\begin{array}
[c]{ll}%
a(x) & \text{when}\quad a^{\prime}(x)>0,\ \\
a^{+}(z_{j-1}) & \text{when}\quad a^{\prime}(x)\leq0,
\end{array}
\right.  . 
\label{eq:atilde} 
\end{align}
{\imagg The values of $\widetilde{a} $ at the breakpoints $z_j$ are unimportant in what follows, but for definiteness we shall require $\widetilde{a}$ to be right continuous at each $z_j$, ($j<N$) and left continuous at $z_N$. }


The
function $\widetilde{c}$ is defined analogously and it is easily verified that
$\widetilde{c^{2}}=\widetilde{c}^{2}$.
From this definition we have the following two propositions. The proof of the
first is very elementary and so omitted.

\begin{proposition} {\imag Under Assumption \ref{ass:coeffs}},    
\label{prop:elem} for  $\jsub
$, 
\[
\widetilde{a}(x) \geq \amin >0,\quad \quad\widetilde{a}^{\prime}(x)\geq0,\quad\quad (\widetilde{a}%
/a)^{\prime}(x)\geq0, \quad \quad\text{for all } \ x\in\mathrm{int}(\tau_{j}), 
\]
{\imag with the analogous result for $\widetilde{c}$. }

\end{proposition}


\begin{proposition}{\imag Under Assumption \ref{ass:coeffs}},
\label{prop:var}
\[
\mathrm{Var}(\widetilde{a}) \ \leq\ \mathrm{Var}(a) , 
\]
{\imag with the analogous result for $\widetilde{c}$.}
\end{proposition}

%

\proof
{\imag Let $\widetilde{\Omega}\ :=\ 
\cup \{\tau_j:a^{\prime}\vert_{\tau_j} > 0 \}  $. } Then, by definition of $\widetilde{a}$,
\[
\int_{-L}^{L}|\partial_{\mathrm{pw}}\widetilde{a}|\ =\ \sum_{\tau_j \subset \widetilde{\Omega}} 
\int_{\tau_j}|a' |\ .
\]
Next, 
we note that if  $a$ is increasing in $\tau_{j}$, then
$[\widetilde{a}]_{z_{j}}=[a]_{z_{j}}$. On the other hand if $a$ is non-increasing
in $\tau_{j}$, then
\[
\lbrack\widetilde{a}]_{z_{j}}=a^{+}(z_{j-1})-a^{+}(z_{j})=(a^{+}(z_{j-1}%
)-a^{-}(z_{j}))+(a^{-}(z_{j})-a^{+}(z_{j}))=\left(  \int_{\tau_{j}}|a^{\prime
}|\right)  +[a]_{z_{j}}.
\]
The result follows on combination of these relations. \ 
\endproof

\begin{notation}
\label{not:bar}
Noting that, by Assumption \ref{ass:coeffs} and definition \eqref{eq:atilde},   $a,c$, $\widetilde{a}$ and $\widetilde{c}$ are all $C_{\mathrm{pw}%
}^{1}$ functions with respect to the partition \eqref{partition}, 
we introduce,  for $\jint$, the quantities  
\[
\alpha_{j}=\max\left\{  \frac{\widetilde{a}^{-}(z_{j})}{\widetilde{a}%
^{+}(z_{j})},1\right\}  ,
\]%
\[
\sigma_{j}=\max\left\{  \frac{(\widetilde{c}^{2})^{-}(z_{j})}{(\widetilde{c}%
^{2})^{+}(z_{j})},1\right\},   
\]
and
\[
\gamma_{j}\ =\ \max\left\{  \frac{a^{+}(z_{j})}{a^{-}(z_{j})},\frac
{(c^{2})^{+}(z_{j})}{(c^{2})^{-}(z_{j})},1\right\}  
.
\]
\end{notation}
This leads us to the definition of the function $q$ which we shall use in
conjunction with Theorem \ref{th:stability1D}.

{\imag \begin{definition}
[the function $q$]\label{defq} We define the increasing sequence of positive
numbers \newline$\{A_{j}:j=1,\ldots,{N}\}$
inductively by
\begin{align}
A_{1}=0\quad\text{and}\quad A_{j+1}\  &  =\ \alpha_{j}\sigma_{j}\gamma
_{j}\left(  \int_{\tau_{j}} \frac{1}{\widetilde{a}\,\widetilde{c}^{2}} + A_{j}%
\right)  ,\quad j=1,\ldots,{\igg N-1},\label{eq:Apos}
\end{align}
Then we define the function $q\in C_{\mathrm{pw}}^{1}[-L,L]$ by 
\begin{align}
q(x)\ & =\ {\widetilde{a}(x)\widetilde{c}^{2}(x)}\left(   \int%
_{z_{j-1}}^{x}\frac{1}{\widetilde{a}(s)\widetilde{c}^{2}(s)}\mathrm{d}%
s  \ +\   A_{j}\right)   ,\quad x\in\tau_{j},  \quad 1 \leq j \leq N.  \label{expq}
 \end{align}

\end{definition}
} 


\begin{lemma}
\label{lem:q}
Under Assumption \ref{ass:coeffs},  
the function $q$ defined in \eqref{expq} is increasing on 
 $[-L, L]$ with $q(-L) = 0$ and satisfies the requirements
\eqref{pwsder} and \eqref{negjumps} of Theorem \ref{th:stability1D}. 

\end{lemma}

\begin{proof}
First note that for $x\in\tau_{j}$, and using Proposition
\ref{prop:elem}, we have
\[
\left(  \frac{q}{a}\right)  ^{\prime}(x)\ =\ \left(  \left(  \frac
{\widetilde{a}}{a}\right)  ^{\prime}(x)\widetilde{c}^{2}(x)\ +\ \left(
\frac{\widetilde{a}}{a}\right)  (x)(\widetilde{c}^{2})^{\prime}(x)\right)
\left(  \int_{z_{j-1}}^{x}\frac{1}{\widetilde{a}(s)\widetilde{c}^{2}%
(s)}\mathrm{d}s+A_{j}\right)  +\frac{1}{a(x)}\geq\frac{1}{a(x)}.
\]
Similarly $(q/c^{2})^{\prime}(x)\geq c^{-2}\left(  x\right)  $. Moreover,
\begin{align}
\left[  \frac{q}{a}\right]  _{z_{j}}  &  =\frac{{\widetilde{a}^{-}%
(z_{j})(\widetilde{c}^{2})^{-}(z_{j})}}{a^{-}(z_{j})}\left(  \int_{\tau_{j}%
}\frac{1}{\widetilde{a}(s)\widetilde{c}^{2}(s)}\mathrm{d}s+A_{j}\right)
-\frac{\widetilde{a}^{+}(z_{j})(\widetilde{c}^{2})^{+}(z_{j})}{a^{+}(z_{j}%
)}A_{j+1}\nonumber\\
&  =\frac{\widetilde{a}^{+}(z_{j})(\widetilde{c}^{2})^{+}(z_{j})}{a^{+}%
(z_{j})}\,\times\\
&  \left[  \left(  \frac{\widetilde{a}^{-}(z_{j})}{\widetilde{a}^{+}(z_{j}%
)}\right)  \left(  \frac{(\widetilde{c}^{2})^{-}(z_{j})}{(\widetilde{c}%
^{2})^{+}(z_{j})}\right)  \left(  \frac{a^{+}(z_{j})}{a^{-}(z_{j})}\right)
\left(  \int_{\tau_{j}}\frac{1}{\widetilde{a}(s)\widetilde{c}^{2}%
(s)}\mathrm{d}s+A_{j}\right)  -A_{j+1}\right]
\end{align}
and the definition \eqref{eq:Apos} ensures this is non-positive. Similarly
$\left[  \frac{q}{c^{2}}\right]  _{z_{j}}\leq0$.

\end{proof}

We now have the main result of this section.

\begin{theorem}
\label{MainStabW11} {\imag Suppose Assumption \ref{ass:coeffs} holds}  and let $u$ solve \eqref{eq:1DB}. 
Then $u$ satisfies the a priori bound
\eqref{MainStab1D}, \eqref{DecCstabI-II},  with
\begin{align}
\label{defQ1}Q \  &  \leq\ {2L} \frac{\amax \cmax^2}{a_{\min}
c_{\min}^{2}} \exp\left(  \frac{2}{a_{\min}}\mathrm{Var}(a) + \frac{2}%
{c_{\min}^{2}} \mathrm{Var}(c^{2})\right)  \ .
\end{align}
In the pure impedance  case (see \eqref{BCprobs}),  the multiplicative factor $2L$ on the right-hand side can be replaced by $L$.  
\end{theorem}

\begin{proof}
We begin by considering the ``Dirichlet-Impedance'' case $\Gamma_D = \{ -L\}$.  
With $q$ as defined above (and since $q(-L) = 0$),   Theorem \ref{th:stability1D} and Lemma \ref{lem:q} then  imply  that
\eqref{MainStab1D} and \eqref{DecCstabI-II} hold,  with
$Q\ =\ q(L)$. 
By induction on
\eqref{eq:Apos}, and using the crucial fact that $\alpha_{\ell}\geq
1,\sigma_{\ell}\geq1,\gamma_{\ell}\geq1$, we obtain
\[
A_{j+1}\ \leq\left(  \prod_{\ell=1}^{j}\alpha_{\ell}\sigma_{\ell}\gamma_{\ell
}\right)  \left(  \int_{-L}^{z_{j}}\frac{1}{\widetilde{a}\,\widetilde{c}%
^{2}}\right)  ,\quad j=1,\ldots N-1,
\]
from which it follows that
\[
A_{N}\ \leq\left(  \prod_{\ell=1}^{N-1}\alpha_{\ell}\right)  \left(
\prod_{\ell=1}^{N-1}\sigma_{\ell}\right)  \left(  \prod_{\ell=1}^{N-1}%
\gamma_{\ell}\right)  \left(  \int_{-L}^{z_{N-1}}\frac{1}{\widetilde{a}%
\,\widetilde{c}^{2}}\right)  .
\]
Thus, from \eqref{expq},
\begin{align}
q(L)\  &  \leq\ {a_{\max}\,c_{\max}^{2}}\left(  \prod_{\ell=1}^{N-1}%
\alpha_{\ell}\right)  \left(  \prod_{\ell=1}^{N-1}\sigma_{\ell}\right)
\left(  \prod_{\ell=1}^{N-1}\gamma_{\ell}\right)  \left(  \int_{-L}%
^{L}\frac{1}{\widetilde{a}\,\widetilde{c}^{2}}\right) \nonumber\\
&  \leq\ 2L\,\frac{a_{\max}\,c_{\max}^{2}}{a_{\min}\,c_{\min}^{2}}\left(
\prod_{\ell=1}^{N-1}\alpha_{\ell}\right)  \left(  \prod_{\ell=1}^{N-1}%
\sigma_{\ell}\right)  \left(  \prod_{\ell=1}^{N-1}\gamma_{\ell}\right)  .
\label{finalA}%
\end{align}
To bound the products in \eqref{finalA}, we appeal to Lemma \ref{lem:tech}.
This, combined with Proposition \ref{prop:var} gives immediately
\begin{equation}
\prod_{\ell=1}^{N-1}\alpha_{\ell}\ \leq\ \exp\left(  \frac{1}{a_{\min}%
}\mathrm{Var}(a)\right)  ,\quad\text{and}\quad\prod_{\ell=1}%
^{N-1}\sigma_{\ell}\ \leq\ \exp\left(  \frac{1}{c_{\min}^{2}}\mathrm{Var}%
(c^{2})\right)  . \label{finalB}%
\end{equation}
Also
\begin{align}
\prod_{\ell=1}^{N-1}\gamma_{\ell}\ \  &  \leq\ \left(  \prod_{\ell=1}%
^{N-1}\max\left\{  \frac{a^{+}(z_{\ell})}{a^{-}(z_{\ell})},1\right\}  \right)
\left(  \prod_{\ell=1}^{N-1}\max\left\{  \frac{(c^{2})^{-}(z_{\ell})}%
{(c^{2})^{+}(z_{\ell})},1\right\}  \right) \nonumber\\
&  \ \leq\ \exp\left(  \frac{1}{a_{\min}}\mathrm{Var}(a)\right)
\exp\left(  \frac{1}{c_{\min}^{2}}\mathrm{Var}(c^{2})\right)
\label{finalC}%
\end{align}
\and the result follows on combining \eqref{finalA}, \eqref{finalB} and
\eqref{finalC}. 

For the Impedance-Dirichlet case we replace $q$ by $q - q(L)$. This function has the same derivative as $q$,   
vanishes at $x = L$ and has maximum 
modulus   $q(L)$ occuring at $x = -L$, so the proof is the same as before. For the Impedance-Impedance case it is natural to replace $q $ by $q - q(L)/2$, which has maximum modulus $q(L)/2$ giving an extra factor of $1/2$ in the estimate.     
\end{proof}


\begin{discussion} {\imagg We finish this subsection with  a short   
 illustration of  how Theorem \ref{MainStabW11} handles both oscillations and jumps in the coefficients $a,c$.  These examples also show that the   
use of   Lemma \ref{lem:tech} to bound the right-hand side of  \eqref{finalA}  
is sharp in terms of the order of its dependence on the variance when 
$a$ or $c$ is  oscillatory, but can be pessimistic in terms of its dependence on $\amax/\amin$ and $\cmax^2/\cmin^2$,  when $a,c$ are not oscillatory. 

\noindent {\em Example 1.}   \ Consider the case when $a = 1$ and $c$ is piecewise constant with respect to the partition \eqref{partition}. 
Suppose $N$ is odd and   set 
\begin{align*}c(x) =  \cmax, & \quad x \in [z_{j-1}, z_j), \quad  j \ \text{odd}, \ j<N \\
   c(x) =  \cmax, & \quad x \in [z_{N-1}, z_N], \quad   j= N  \\
c(x) =  \cmin, & \quad x \in [z_{j-1}, z_j), \quad  j \ \text{even}, \ \end{align*} 
where  $\cmax > \cmin > 0$. Then  it is easy to see that (with the definitions as in Notation \ref{not:bar}), for each $j = 1, \ldots, N-1$,  $\alpha_j = 1$  and 
\begin{align*} \sigma_j & = \left(\frac{\cmax}{\cmin}\right)^2 \quad \text{and} \quad \gamma_j = 1,   
\quad  \text{when} \ j \quad \text{is odd }\\   
 \sigma_j & = 1 \quad \text{and} \quad \gamma_j  = \left(\frac{\cmax}{\cmin}\right)^2 , \quad   \text{when} \  j \quad \text{is even}\ .
\end{align*}
Hence the estimate \eqref{finalA} yields 
$$Q \ \leq  2L \left(\frac{\cmax^2}{\cmin^2}\right)^N.   \ $$
In this case $\Var(c^2) = (N-1)(\cmax^2 - \cmin^2)$. Thus,  
$\Var(c^2)$   grows linearly in $N$ which implies that the bound on $Q$  grows exponentially in $\Var(c^2)$. 
Similar results are implied by the estimates in \cite{FreletDiss} and \cite{BaChGo:17}.   
} 

\noindent {\em Example 2.}\  Consider the case when $c= 1$ and $a$ is oscillatory, given by 
$$ a(x) = 2 + \sin(m \pi x/L), \quad x \in [-L,L] , $$
where $m$ is chosen to be an even   positive integer. 
Then $a'$ changes sign at the points $$z_j := (j-m -1/2)L/m, \quad j = 1,\ldots, 2m.$$ 
These form the interior points of the partition \eqref{partition}, so that $N = 2m +1$, 
{\imagg and we set  $z_0 := -L,$ and $z_{2m + 1} := L$.  
Recall the definition of $\alpha_j, \sigma_j$ and $\gamma_j$ from Notation \ref{not:bar}.  It is easily seen that   $\sigma_j = \gamma_j = 1$, for $j = 1,\ldots 2m$.}    
Moreover $a(z_0) = a(z_{2m+1}) = 2$ and 
$$ a(z_j) = \left\{ \begin{array}{rl} 1,  & j \ \text{is even} \\
3,  & j\  \text{is odd} \end{array} \right. 
$$
Since $a$ switches from decreasing to increasing at $z_j$ with $j$ even, we have 
$$\alpha_j = \frac{a(z_{j-1})}{a(z_j)} \ =  \ 3 ,   \quad \text{when} \ j \ 
\text{is even } \quad \text{and} \quad    \alpha_j = 1 \quad \text{when } \ j \ \text{is odd}. $$
Hence 
the estimate \eqref{finalA} yields 
\begin{align}
Q \  
&  \leq\  2L\,\frac{a_{\max}}{a_{\min}}
\left(
\prod_{\stackrel{\ell={1}}{\ell \text{even}} } ^{2m} \alpha_\ell \right) =  
6L \, 3^{m}       .
\nonumber%
\end{align}
noting that  $\Var(a)$ grows linearly with $m$, we see that again  the above estimate grows exponentially with the order of the variance.  

{\imagg \noindent {\em Example 3.} Consider the non-oscillatory case when both 
$a$ and $c$ are  monotonic   (decreasing or increasing) functions on $[-L, L]$. 
 Then there are no interior points in the partition \eqref{partition}, and $z_0 = -L$, $z_1 = L$. Then using \eqref{finalA} to estimate $Q$ directly we would obtain $$ Q \ \leq \ 2L \frac{\amax}{\amin} \frac{\cmax^2}{\cmin^2},  $$
whereas  the estimate \eqref{defQ1} (which made use of Lemma \ref{lem:tech}) would be somewhat worse:  
$$ Q \ \leq \ 2L \frac{\amax}{\amin} \frac{\cmax^2}{\cmin^2} \exp\left(2 \frac{\amax}{\amin} + 2\frac{\cmax^2}{\cmin^2}  - 4\right).  $$ }
  

\end{discussion}
\subsection{On the sharpness of the estimates}

\label{Exbad}

In this section we will show by means of a family of examples that the bound
in Theorem \ref{MainStabW11} is sharp in the sense that the exponential growth
of the stability constant in $\mathrm{Var}(c)$ can be realised in numerical
simulations. The derivation of these 
``{\em nearly unstable}''  examples can be found in \cite{Torres17}
and further classes of examples which may also serve as benchmark problems are
derived in \cite{Torres18}.
Each member of our family has the general (strong) form:
\begin{equation}
\left.
\begin{array}
[c]{lll}
& -u^{\prime\prime}-\left(  \frac{\omega}{c}\right)  ^{2}u & =0\quad\text{in
}\Omega=\left(  -1,1\right)  ,\\
\text{with} & \left(  -u^{\prime}-\operatorname*{i}\frac{\omega}{c}u\right)
(-1)\  & =\ g_{1}\\
\text{and} & \left(  u^{\prime}-\operatorname*{i}\frac{\omega}{c}u\right)
(1)\  & =\ g_{2}.
\end{array}
\right\}  \label{Helmpwconst}%
\end{equation}
The family will be specified by a countably infinite sequence of frequencies
$\omega_{m}$, where $m$ ranges over the positive even integers, and a
corresponding sequence of piecewise constant wave speeds $c_{m}$ with
increasing numbers of jumps.
For this problem it is possible to write the analytic solution in each
subinterval on which $c_{m}$ is constant as a linear combination of left- and
right-travelling waves, with coefficients determined by the boundary data
$g_{1}$, $g_{2}$, the frequency $\omega_{m}$ and the wave-speed $c_{m}$. The
analytic solution is then equivalent to solving a system of linear equations,
the properties of which can be studied by symbolic manipulation. Using this
approach,  and looking for situations in which this system becomes{
ill-conditioned}, the following unstable case has been derived.




Let $r\in(0,1)$ and let $m$, be an even positive integer. For each $m$ we
specify the frequency
\begin{equation}
\omega_{m}=\frac{\pi}{2}(1-r+m),\quad m=2,4,6,\ldots.\label{eq:deffreq}%
\end{equation}
To specify the wave speed $c_{m}$ we choose a partition of $[-1,1]$ with
$2m+1$ subintervals of the form:
\begin{equation}
-1=x_{0}^{m}<x_{1}^{m}<\ldots x_{2m+1}^{m}=1.\label{defxi}%
\end{equation}
For each $\ell=1,\ldots,2m+1$ we define the wave speed on $\tau_{\ell}%
^{m}=(x_{\ell-1}^{m},x_{\ell}^{m})$ to be
\begin{equation}
c_{m,\ell}=\left\{
\begin{array}
[c]{ll}%
1-r & \text{when}\quad\ell\quad\text{ is odd},\\
1+r & \text{when}\quad\ell\quad\text{ is even}\ .
\end{array}
\right.  \label{eq:defc}%
\end{equation}
The partition $\left(  x_{\ell}^{m}\right)  _{\ell=0}^{2m+1}$ is fixed by
setting
\[
(x_{\ell}^{m}-x_{\ell-1}^{m})=\left\{
\begin{array}
[c]{ll}%
\frac{c_{m,\ell}}{1-r+m} & \text{when}\quad\ell\not =m+1\ ,\\
& \\
\frac{2c_{m,\ell}}{1-r+m} & \text{when}\quad\ell=m+1\ .
\end{array}
\right.
\]
Combining this prescription with $x_{0}:=-1$, it is easy to see that the
resulting partition is of the form \eqref{defxi}.

Since the right-hand side in the first equation in \eqref{Helmpwconst}
vanishes, {\imagg by writing down  \eqref{eq:1DB} with $v = u$ and taking the real part,  we have} 
\[
\int_{\Omega}{\igg |u ' |^{2}}\ =\ \int_{\Omega}\left(  \frac{\omega}{c}\right)
^{2}|u|^{2},
\]
which implies $\Vert u^{\prime}\Vert\ =\ \Vert(\omega/c)u\Vert$ and $\Vert
u\Vert_{\mathcal{H},a,c}=\sqrt{2}\Vert u^{\prime}\Vert$ . In the tables below
we present computed values of $\Vert u^{\prime}\Vert$ corresponding to varying
choices of $r$ and $m$. The computations are done by applying the standard
continuous linear finite element method to \eqref{Helmpwconst}. For given $m$,
we construct an initial piecewise uniform grid with $800$ equal elements on
each subinterval $\tau_{\ell}^{m}$, $\ell=1,\ldots,2m+1$. On this grid we
compute the finite element solution $u_{h}$ and then $\Vert u_{h}^{\prime
}\Vert$. All integrations in the implementation are done exactly. Then we
repeat this calculation using a sequence of six additional uniform
refinements, the finest one having $(800\times2^{6})\times(2m+1)=51200\times
(2m+1)$ elements in each $\tau_{\ell}^{m}$.
Provided the computed value of $\Vert u_{h}^{\prime}\Vert$ does not change (in
its first four significant figures) in the final three of these seven
successive refinements, then that value is recorded as the true value of
$\Vert u^{\prime}\Vert$ (to four figures).

Computations are done in \texttt{matlab} and the required linear systems are
solved using the standard sparse backslash. Overall, the linear systems being
solved are quite ill-conditioned and the computation of $u_{h}$ can become
unstable on the finest meshes for the largest values of $m$, especially when
$r$ is relatively close to $1$. In some of our experiments we failed to achieve  
convergence to  four significant figures. Such results are labelled with a * in
the tables.

Table \ref{tab:1} illustrates the properties of $\Vert u^{\prime}\Vert$ as $m$
and $r$ vary. In the column labelled $\kappa$ we give the estimated condition
number of the system matrix on the finest grid (computed using the
\texttt{matlab} function \texttt{condest}.



\begin{table}[h]
\begin{center}%
\begin{tabular}
[c]{|l||l|l||l|l||l|l|}\hline
& \multicolumn{2}{|c||}{$r = 0.4$} & \multicolumn{2}{|c||}{$r = 0.5$} &
\multicolumn{2}{|c|}{$r = 0.6$}\\\hline
$m$ & $\Vert u^{\prime}\Vert$ & $\kappa$ & $\Vert u^{\prime}\Vert$ & $\kappa$
& $\Vert u^{\prime}\Vert$ & $\kappa$\\\hline
2 & 7.742(-1) & 5.46(+10) & 8.498(-1) & 8.21(+10) & 9.642(-1) & 1.34(+11)\\
4 & 1.313 & 3.46(+11) & 1.845 & 8.22(+11) & 2.789 & 2.31(+12)\\
6 & 2.538 & 1.98(+12) & 4.588 & 7.63(+12) & 9.238 & 3.75(+13)\\
8 & 5.180 & 1.10(+13) & 1.203(+1) & 6.94(+13) & 3.225(+1) & 6.03(+14)\\
10 & 1.088(+1) & 6.06(+13) & 3.247(+1) & 6.26(+14) & 1.16(+2) * & 9.63(+15)\\
12 & 2.329(+1) & 3.31(+14) & 8.9(+1)* & 5.64(+15) & 4.2(+2)* &
1.44(+17)\\\hline
\texttt{grad} & 0.34 &  & 0.46 &  & 0.61 & \\\hline
\end{tabular}
\end{center}
\caption{Values of $\Vert u^{\prime}\Vert$ for problem \eqref{Helmpwconst},
with $\omega^{m}$ given in \eqref{eq:deffreq} and $c^{m}$ given in
\eqref{eq:defc}. The data in \eqref{Helmpwconst}  is chosen as $g_{1} = 0$, $g_{2} = 1$ and $\kappa$
denotes the condition number of the system matrix in each case. Values
labelled with $^{*}$ have converged to less than four significant figures.
\texttt{grad} indicates the gradient of the linear least squares fit to the
data {\igg $(m, \log(\Vert(u^{m})^{\prime}\Vert))$}. }%
\label{tab:1}%
\end{table}

From these computations we see clearly the blow up of the Helmholtz energy $2
\Vert u^{\prime}\Vert$ as $m$ increases, with the rate of blow-up increasing
as $r$ increases. The results can be seen to reflect the theoretical worst
case bound
as follows.

Since $f=0$ and $\Vert g\Vert_{N}=1$, Theorem \ref{MainStabW11} gives the
bound
\[
\Vert u^{\prime}\Vert\ \leq\frac{1}{\sqrt{2}}C_{\mathrm{stab}}%
^{\operatorname{II}}\sqrt{Q_{\ast}}%
\]
with
\[
C_{\mathrm{stab}}^{\operatorname{II}}\ =\ 2\sqrt{\frac{3(1+r)}{2(1-r)}%
+1},\quad\text{and}\quad {\imag Q_{\ast}=2\left(\frac{1+ r} {1-r} \right)^2 
\exp\left(  4m \frac{(1+r)^2}{(1-r)^4}\right) } \ .
\]
Thus, with $u^{m}$ denoting the solution of problem \eqref{Helmpwconst} for
each $m$, we have
{\imag \begin{align}\label{eq:compest}
\log\Vert u^{\prime}\Vert\ \leq\ 2m \frac{(1+r)^2}{(1-r)^{4}} +
\log\left(
C_{\mathrm{stab}}^{\operatorname{II}} \frac{1+r } {1-r}\right) \ .
\end{align}}
Extrapolation on the data in Table \ref{tab:1} indicates that $\log
(\Vert(u^{m})^{\prime}\Vert)$ grows approximately linearly with $m$. The
gradient of the linear least squares fit to the data $(m,\log(\Vert
(u^{m})^{\prime}\Vert))$ is indicated in the last
line of Table \ref{tab:1} 
and is seen to grow weakly as $r$ increases, numerically supporting the estimate \eqref{eq:compest}.

The instability illustrated above is sensitive to the boundary data
$(g_{1},g_{2})$. In Table \ref{tab:2} we illustrate two different cases, one
stable and one not.
\begin{table}[h]
\begin{center}%
\begin{tabular}
[c]{|l||l|l|}\hline
$m$ & $g_{1}= 1 = g_{2}$ & $g_{1}= 2, \ g_{2} = 0.5$\\\hline
$2$ & 4.677(-1) & 1.520\\
$4$ & 3.480(-1) & 4.198\\
$6$ & 2.887(-1) & 1.386(+1)\\
$8$ & 2.520(-1) & 4.838(+1)\\
$10$ & 2.26(-1)* & 1.70(+2)*\\
$12$ & 2.1(-1)* & 6.30(+2)*\\\hline
\end{tabular}
\end{center}
\caption{Values of $\Vert u^{\prime}\Vert$ for problem \eqref{Helmpwconst},
with $\omega^{m}$ given in \eqref{eq:deffreq} and $c^{m}$ given in
\eqref{eq:defc}, with $r = 0.6$. }%
\label{tab:2}%
\end{table}
In Table \ref{tab:3}   we study the sensitivity of the instability to small changes
in the data. We consider the same data as for the problem in Table
\ref{tab:1}, except that we perturb the mesh point $x_{k+1}$ by a small
parameter $\varepsilon$, i.e. the mesh is as in \eqref{defxi} except that
\begin{align}
\label{eq:pert}x_{m+1} = x_{m+1} + \varepsilon.
\end{align}
Computations for various $m$ and $\varepsilon$ are given in Table \ref{tab:3}.
Computations with $\varepsilon\leq10^{-7}$ and $m \geq14$ are not sufficiently
convergent and hence are not reported.   {\igg Table \ref{tab:3} indicates that the unstable case constructed in 
\eqref{eq:deffreq}, \eqref{defxi}, \eqref{eq:defc} is very delicate. 
By adding a perturbation of about $10^{-6}$  to the location of the   central jump point  $x_{m+1}$, we turn an unstable problem 
to a stable one.}

\begin{table}[h]
\begin{center}%
\begin{tabular}
[c]{|l||l|l|l|l|l|l|l|l|}\hline
$m\backslash\varepsilon$ & $0$ & $10^{-9} $ & $10^{-8} $ & $10^{-7} $ &
$10^{-6} $ & $10^{-5} $ & $10^{-4} $ & $10^{-3} $\\\hline
$6$ & 4.59 & 4.59 & 4.59 & 4.59 & 4.59 & 4.578 & 3.829 & 0.7256\\
$8$ & 12.03 & 12.03 & 12.03 & 12.03 & 11.99 & 9.49 & 1.547 & 0.2603\\
$10$ & 32.47 & 32.47 & 32.47 & 32.3* & 24.5* & 3.73* & 0.4200 & 0.1927\\
$12$ & 89.35 & 89.28 & 88.9* & 65* & 9.57* & 0.978* & 0.1982 & 0.1735\\
$14$ &  &  &  &  & 2.57* & 0.3030 & 0.1629 & 0.1608\\
$16$ &  &  &  &  & 0.72* & 0.1644 & 0.1509 & 0.1508\\
$18$ &  &  &  &  & 0.244* & 0.1437 & 0.1424 & 0.1424\\
$20$ &  &  &  &  & 0.1466 & 0.1354 & 0.1353* & 0.1353*\\\hline
\end{tabular}
\end{center}
\caption{Values of $\Vert u^{\prime}\Vert$ for problem \eqref{Helmpwconst},
with $\omega^{m}$ given in \eqref{eq:deffreq} and $c^{m}$ given in
\eqref{eq:defc} with $r = 0.5$. Mesh is as in \eqref{defxi} with perturbation
\eqref{eq:pert} . }%
\label{tab:3}%
\end{table}

\begin{figure}[ht]
\begin{tabular}{rl}
\includegraphics[scale=0.5]{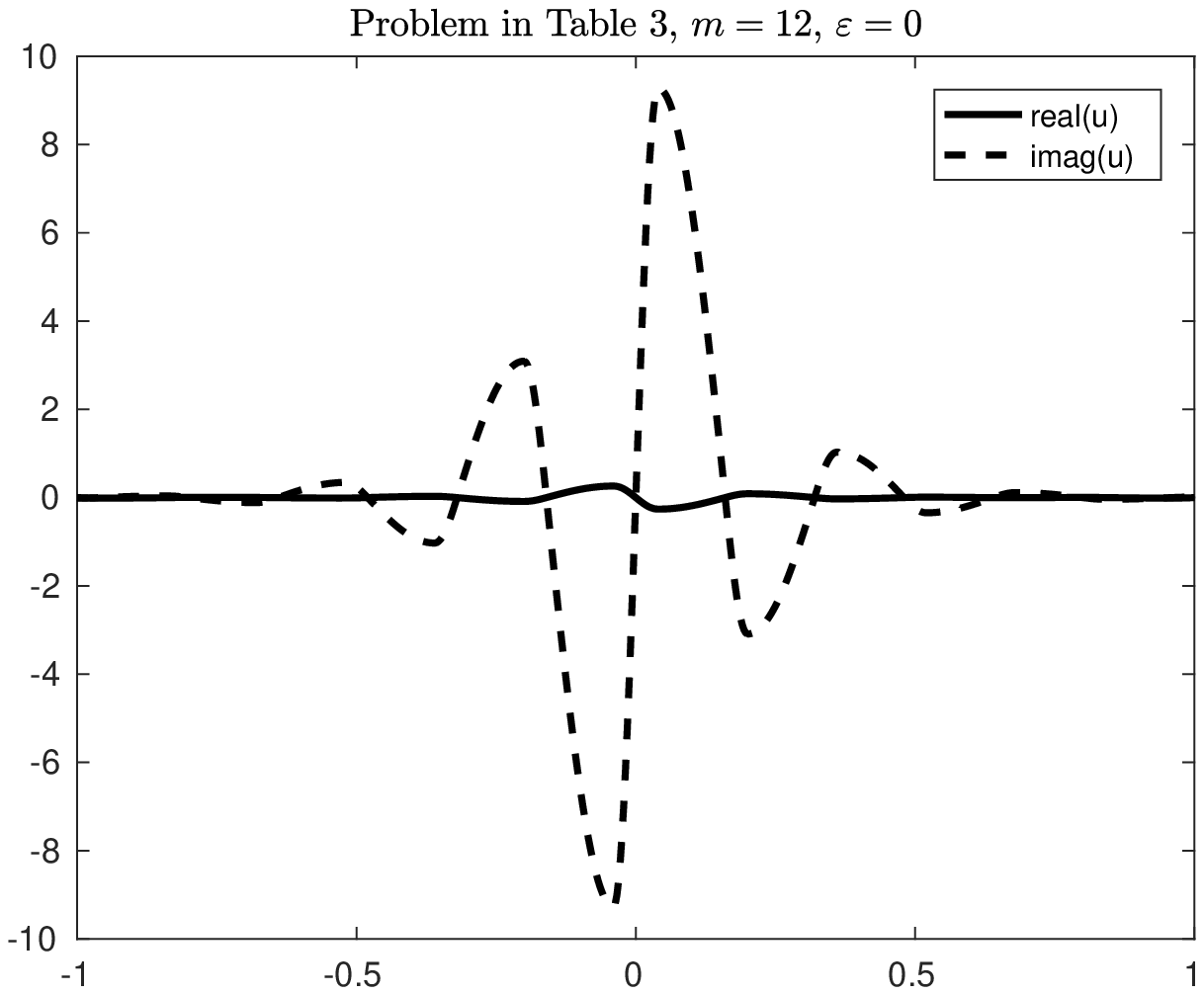} & \includegraphics[scale=0.5]{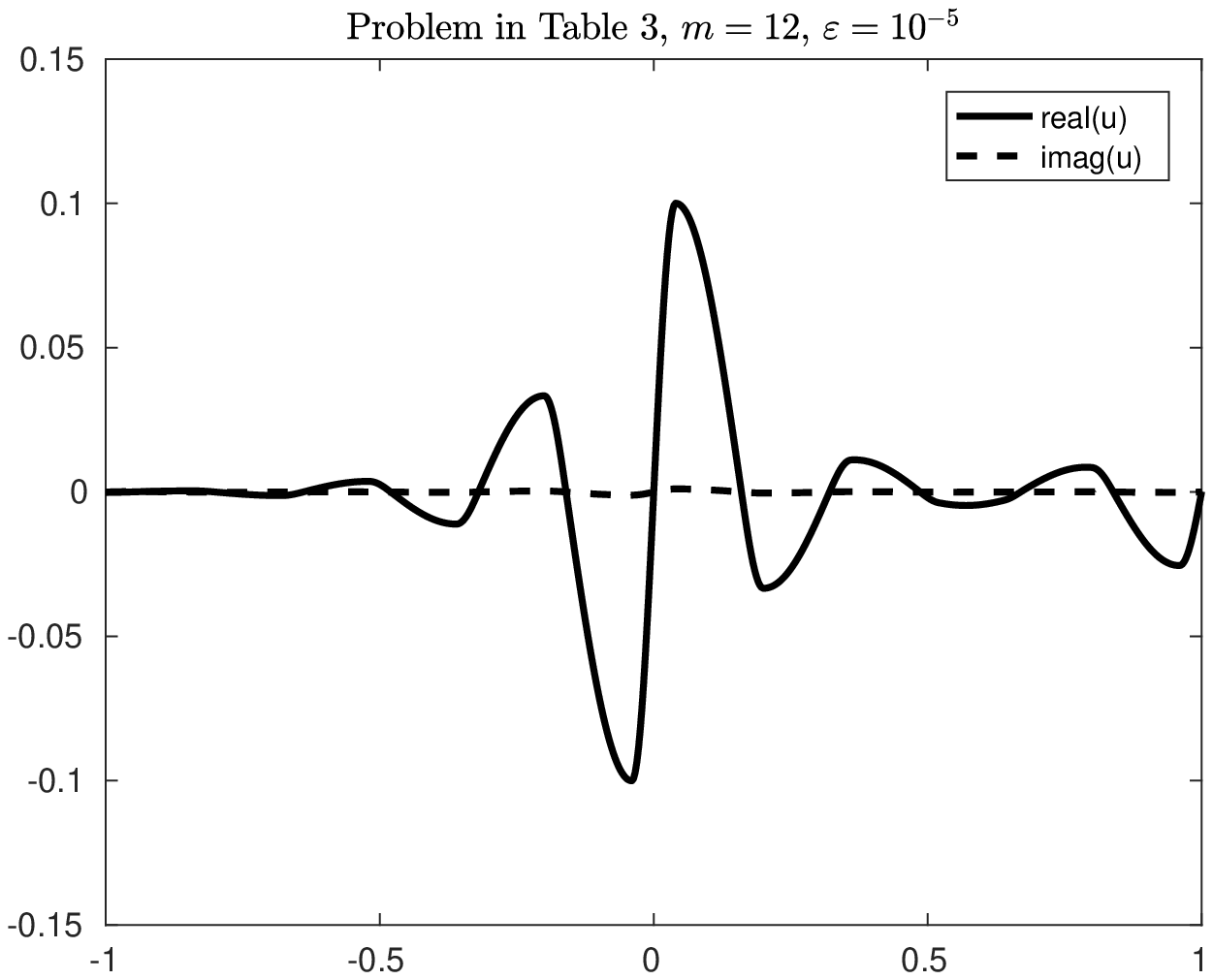}
\end{tabular} 
\caption{\label{fig:2} Graphs of the real and imaginary parts of the solution $u$ to the problem computed in Table \ref{tab:3} with $m = 12$ and $\varepsilon = 0$ (left) and $\varepsilon = 10^{-5}$ (right).  }
\end{figure} 

{\imag In order to illustrate the substantial effect a small perturbation can have on
the solution near an instability we give in Figure \ref{fig:2}  
graphs of the solution to the problem studied in Table
\ref{tab:3} for the cases $\varepsilon=0$ and $\varepsilon=10^{-5}$. Note the
substantial difference in the vertical scales in these two graphs, while the
data is only different $10^{-5}$.
}

\section{Appendix}

\begin{lemma}
\label{lem:tech} Suppose $f \in C_{\mathrm{pw}}^{1}[-L,L]$ with break points
as in \eqref{partition}, so that, on each $\tau_{j}$, either $f^{\prime}(x)
\leq0$ or $f^{\prime}(x) > 0$ . Suppose also $f(x) \geq f_{\min} > 0$ for all
$x \in[-L,L]$. Then
\begin{align}
\label{firstprod}\prod_{\ell= 1}^{N-1} \max\left\{  \frac{f^{\pm}(z_{\ell}%
)}{f^{\mp}(z_{\ell})}, 1 \right\}  \ \leq\ \exp\left(  \frac{1}{f_{\min}}
\mathrm{Var} (f)\right)  \ .
\end{align}

\end{lemma}

\begin{proof}
We restrict the proof to the case of   $f^{+}$ in the numerator and
$f^{-}$ in the denominator on the left-hand side of \eqref{firstprod}. The other case is 
analogous. 
Let the left hand side of this inequality be denoted $C$. Then
\begin{align*}
\log(C)  &  \ = \sum_{\overset{\ell= 1}{f^{+}(z_{\ell}) > f^{-}(z_{\ell})}%
}^{N-1} \log\left(  \frac{f^{+}(z_{\ell})}{f^{-}(z_{\ell})} \right) \\
&  = \sum_{\overset{\ell= 1}{f^{+}(z_{\ell}) > f^{-}(z_{\ell})}}^{N-1}
\log\left(  1 + \frac{f^{+}(z_{\ell}) - f^{-}(z_{\ell})}{f^{-}(z_{\ell})}
\right) \\
&  \ < \frac{1}{{f_{\min}}} \sum_{\overset{\ell= 1}{f^{+}(z_{\ell}) >
f^{-}(z_{\ell})}}^{N-1} {(f^{+}(z_{\ell}) - f^{-}(z_{\ell}))} \ \leq\frac
{1}{{f_{\min}}} \sum_{\ell= 1}^{N-1} \vert[f]_{z_{\ell}} \vert \leq 
\ \frac{1}{f_{\min}} \mathrm{Var}_{[z_{0},L]} (f) \ .
\end{align*}
\end{proof}

{\noindent\textbf{Acknowledgement} \ We are grateful to the Hausdorff Research
Institute for Mathematics in Bonn for Visiting Fellowships in their 2017
Trimester Programme on Multiscale Methods, during which part of this work was
carried out. The first author also thanks the Institut f\"{u}r Mathematik at
the University of Z\"{u}rich for financial support. We would like to thank
Euan Spence for very useful discussions. } 


\end{document}